\documentclass{article}
\usepackage{commath}
\usepackage{relsize}
\usepackage{mathdots}
\usepackage{amsmath,amsthm}
\usepackage{amsfonts}
\usepackage[all]{xy}
\usepackage{graphicx}
\allowdisplaybreaks
\newcommand{\esssup}{\mathop{\rm{ess}\,\sup}}
\newcommand{\argmax}{\mathop{\arg\max}}
\newcommand{\argmin}{\mathop{\arg\min}}

\newtheorem{Ass}{Assumptions}
\newtheorem{Rem}{Remark}

\newtheorem{Pro}{Problem}[section]
\newtheorem{Def}{Definition}[section]
\newtheorem{Thm}{Theorem}[section]
\newtheorem{Lem}{Lemma}[section]

\newtheorem{Cor}{Corollary}[section]

\begin{document}
	\pagenumbering{gobble}
	\title{Finite Horizon Impulse Control of Stochastic Functional Differential Equations}
	\author{J. Jönsson and M. Perninge}
	\maketitle
\begin{abstract}
	In this work we show that one can solve a finite horizon non-Markovian impulse control problem with control dependant dynamics. This dynamic satisfies certain functional Lipschitz conditions and is path dependent in such a way that the resulting trajectory becomes a flow. 
\end{abstract}

\section{Introduction}
The finite horizon impulse control problem is a type of optimal stochastic control problem. Admissible controls consist of an increasing sequence of stopping times $\tau_{i}$ and a corresponding sequence of random variables $\xi_{i}$ representing  impulses affecting an underlying state. Apart from such a control, the evolution of this state is usually determined by a stochastic differential equation where the noise stems from a Brownian motion. The underlying state, together with the control, in turn affects a performance functional which is to be maximised.

This performance functional has a running part, depending continuously on the underlying state, as well as a negative impulse part rendering a "cost" depending on the choice of impulses that we allow to affect the state during its evolution.

More explicitly, given $\nu=(\tau_{i},\xi_{i})_{i\in \mathbb{N}}$, the underlying state follows the dynamic
\begin{equation}
	\begin{aligned}
		&dX^{\nu}_{t}=a(t,X^{\nu}_{t})dt+b(t,X^{\nu}_{t})dB_{t} \quad \tau_{i}< t < \tau_{i+1}\\
		&X^{\nu}_{\tau_{i}}=\Gamma(X^{\nu}_{\tau_{i}-},\xi_{i}),
	\end{aligned}
\end{equation}

where $B_{t}$ is a Brownian motion. The problem is to find $\nu$ such that

\begin{equation}
	\begin{aligned}
		{J}(\nu)=\mathbb{E}[\int_{0}^{T}f(s,X^{\nu}_{s})ds-\sum_{i} \ell(X^{\nu}_{\tau_{i}},\xi_{i},\tau_{i})]
	\end{aligned}
\end{equation}

is maximised, given a finite $T$.
\newpage
When faced with these types of problems there are a few natural questions that arise, e.g\\

\begin{itemize}
	\item Does an optimal control exist?\\
	
	\item Is it possible to explicitly compute this control for a given problem?
	
\end{itemize}
One of the most well known methods for handling these questions is dynamic programming. This method consists of several different parts, e.g Bellman's principle, Bellman's equation and the backward induction algorithm.

Bellman's principle could be considered the foundation in dynamic programming, as this enables us to divide the problem into sub-problems. If a given control problem does not possess this structure, solving it, analytically or by means of numerical methods, requires a different approach than dynamic programming. 

In regard to impulse problems, there are several frameworks in which Bellman's principle
can be expressed, e.g obstacle problems, Snell envelopes and  Reflected Backward Stochastic Differential Equations. To clarify this, as well as how these approaches developed, we proceed with a brief summary of a subset of the vast number of contributions that has developed this field.

\subsection{A brief summary of the two main approaches to impulse control and their history}
The mathematical tools that are used nowadays to solve impulse problems resemble the ones used to solve a closely related problem, the optimal stopping problem. That these two problems are related, is at least intuitively clear, since an impulse problem involves a sequences of stopping times when we alter the system.

In the early 1950’s, inspired by Wald’s work\cite{wald} and using Doob’s theory of martingales, Snell \cite{snell} was the first to explicitly formulate and solve an optimal stopping problem. This was done by means, which we today consider to be part of the so called "Martingale approach" to optimal stopping.

Around the latter half of the 1950’s, another method for solving stopping problems was proposed. This method related the stopping problem to a obstacle problem in the way today known as the "principle of smooth fit" or "high contact principle". While using this approach one has to assume that the underlying dynamic has a Markovian structure. Methods stemming from this approach are thus today referred to as the "Markovian approach". The exact origins of the many different aspects of this method is unclear. Some of the important early contributors to the approach are listed in \cite{opt}.

In the early 1960’s, Dynkin \cite{dyn} characterized the value function of the optimal stopping problem as the smallest superhamonic majorant of the pay-off function as well as relating the two approaches. From this point onwards, an intense study of the stopping problem commenced and various generalized methods were introduced to handle it, e.g variational inequalities and viscosity solutions. It seems that during this time the seeds were sown for the modern mathematical treatment of the impulse problem. In particular, A. Bensoussan and J.L. Lions\cite{BL}, who used functional analysis to solve a stopping problem, went on to formulate and solve an impulse problem in the same framework. A few years later, motivated by applications in economics, Brekke and Øksendal\cite{bo} were able to relax some of the existing conditions on the data. This was achieved by working with the high contact principle and hence did not rely on weak derivatives, in turn making their approach better suited for applications. Both Bensoussan and Lions and Brekke and Øksendals methods are related to obstacle problems but expresses this in different ways.

In the middle of the 1990’s, El Karoui et al.\cite{elk} used the Snell envelope to solve a stopping problem of Lagrange-Meyer type. Their framework was later extended to handle optimal switching, first with two modes\cite{HaJe} and later a finite number of modes\cite{DJ}. These extensions were made possible by the shape of the Lagrange-Meyer pay-off, and entailed establishing Bellman's principle in terms of Snell envelopes. In this paper we will follow this line of work by establishing a Bellman principle in the setting of a path-dependant impulse problem where the control enters the volatility term $b$.

The above summary of contributions, to both impulse and stopping problems, is far from complete. The purpose of including it, is mainly to highlight the fact that both problems can be treated in two separate mathematical languages, in turn putting our result into a broader context.

\subsection{Our contribution and approach}
Most of the work on impulse problems has been carried out using the Markovian approach. The theorems that have been proved are in most cases so called verification theorems. In such theorems the existence of an optimal control often relies on the existence of a solution to a certain equation, the latter of which also expresses Bellman's principle but in terms of obstacle problems. In order to use the results one furthermore has to find this solution. These requirements are in general hard to fulfill, in particular proving existence of a sufficiently regular solution.

Within the Markovian framework there is recent work \cite{tysk} that provides a verification theorem which has less restrictive assumptions on the solution. In particular, their result only requires continuity of the solution and also proves that such a solution exists. Besides reducing the complexity of the assumptions in \cite{tysk}, our method allows, in the Markovian setting, for negative costs, less regular data in the cost functional and need not the assume admissibility of the optimal control.

In the non-Markovian setting, Djehiche, Hamadène and Hdhiri \cite{DHH} used families of interconnected Snell envelopes to characterize and prove existence of an optimal control. However, their formulation of the impulse problem differs somewhat from ours and that of e.g \cite{OS}.

Our setup can briefly be characterized as follows. Our state will have a dynamic similar to $(1)$, but we shall allow the coefficients to depend on the entire path. In contrast to \cite{DHH}, in addition to how the control acts on the state and how the latter depends on the former, we only assume that our control set is a compact subset of $\mathbb{R}^n$.

This will force us to overcome a different set of problems, in particular, our so called verification family will be different and we will have to use the concept of stochastic flows to obtain Lyapunov-like stability of our state in connection to our impulses. To the best of our knowledge there are no results on the non-Markovian impulse problem as formulated by us.

We mention the results of \cite{P1}, which we extend by considering impulse controls rather than switching controls and a more general trajectory dependence. The results are also related to the work in \cite{P2} where an abstract impulse control problem is solved.

The paper is structured as follows. The next section contains the formulation of the problem as well as the necessary definitions. The third section treats the underlying state and how it behaves in relation to impulses. Section $4$ contains our verification theorem followed by Section $5$ where we show that our assumptions are sufficient to guarantee existence of a solution. The last section contains an application of our results. In particular, impulse control of stochastic delay differential equations(SDDEs), which are necessarily non-Markovian. We also provide numerical calculations on a well known system with delays using recent proposed methods.

\section{Formulation, Assumptions and Auxiliary results}
Once and for all we fix a given filtered probability space $(\Omega,\mathcal{F},\mathbb{P},\{\mathcal{F}_{t}\})$. The filtration is the natural filtration of a $n$-dimensional Brownian motion $B_{t}$ defined on this space and is in addition completed with the $\mathbb{P}$-null sets.

Furthermore, we let $\mathcal{P}$ denote the set of real-valued $ \{\mathcal{F}_{t}\}$-progressive processes. For $p>1$ we consider the following subsets of $\mathcal{P}$

$\mathcal{H}^{p}=\{X;\mathbb{E}[\int_{0}^{T}\abs{X_{t}}^{p} dt] < \infty  \}$

$\mathcal{S}^{p}=\{X;\mathbb{E}[\sup_{t\in[0,T]} \abs{X_{t}}^{p}]< \infty \, \text{and a.s. cadlag}\}$

$\mathcal{S}^{p}_{c}=\{X;\mathbb{E}[\sup_{t\in[0,T]} \abs{X_{t}}^{p}]< \infty \, \text{and a.s. continuous}\}$

Moreover we let $\mathcal{D}$ denote the space of all  cadlag functions and $\norm{\cdot}$ denote the standard Euclidean norm on $\mathbb{R}^{n}$

The set of $\{\mathcal{F}_{t}\}$-stopping times after some stopping time $\tau$ will be denoted $\mathcal{T}_{\tau}$ i.e all stopping times $\hat{\tau}$ such that $\tau \le \hat{\tau} \le T, \mathbb{P}$-a.s. Moreover, $\mathcal{F}_{\tau}$ will denote the sigma algebra at a stopping time $\tau$.

Throughout the text $C$ will denote a generic constant and we will use $C_{p}$ for constants for which we wish to indicate the origin, where $p$ is the relation to some $L^p$ space, if present, e.g if Burkholder-Davis-Gundy is used. Below constants $K_{i}$ will be introduced which will represent bounds related to assumptions on the dynamic.

A control $\nu$ is a sequence of pairs $(\tau_{i}$, $\xi_{i})_{i\in \mathbb{N}}$ where $\tau_{i}$ is an increasing sequence of $\{\mathcal{F}_{t}\}$-stopping times and $\xi_{i}$ is a sequence of $\mathcal{F}_{\tau_{i}}$-measurable real-valued random variables that take values in a compact set $U \subset \mathbb{R}^m$ according to the magnitude of the impulse at $\tau_{i}$. Any control that satisfy $\lim \tau_{i}=T, \ \mathbb{P}$-a.s,  is called admissible and we denote the set of all such controls  $\mathcal{A}$. The subset of $\mathcal{A}$ such that $\mathbb{P}(\omega ; \tau_{i}(\omega)<T$ for all $i\ge 0)=0$ are called the finite controls and is denoted $\mathcal{A}_{f}$. Moreover this subset contains the following subsets $\mathcal{A}^{k}_{f}=\{\nu \in \mathcal{A}_{f}:\tau_{k+1}=T \}$. We introduce the following operation on the controls,
\begin{Def}
	Given $\nu_{1}= (\tau_{i,1}, \xi_{i,1})_{i\in \mathbb{N}}\in \mathcal{A}_{f}$ and $\nu_{2}= (\tau_{i,2}, \xi_{i,2})_{i\in \mathbb{N}}\in \mathcal{A}$  we set
	\begin{equation}
		\nu_{1}\circ \nu_{2}:=(\tau_{1,1},\xi_{1,1}, \ldots \tau_{I,1},\xi_{I,1},\tau_{I,1}\vee\tau_{1,2},\xi_{1,2},\ldots \tau_{I,1}\vee\tau_{i,2},\xi_{i,2}\ldots),
	\end{equation}
\end{Def} 
where $I(w)=\min\{i;\tau_{i}(\omega) \ge T\}$. For readability, we will denote a large number of compositions of this operation by $\bigcirc_{i=0}^{n}\nu_{i}=\nu_{1}\circ \cdots \circ \nu_{n}$. Note that in order for such compositions to be well defined, the first $n-1$ controls has to belong to $\mathcal{A}_{f}$

The coefficients and the jumps of the state dynamics will be subject to the following constraints,

\begin{Ass}\label{A1}
	$(i)\,a,b:[0,T] \times \Omega \times \mathcal{D}^n \rightarrow \mathbb{R}^{l\times q}$, $a(t,\omega,0)$ and $b(t,\omega,0)$ are a.s continuous in $t$ and the components satisfy 
	
	\begin{equation}
		\mid a_{i}(t,\omega,\{X_{s}\}_{s \le t})-a_{i}(t,\omega,\{Y_{s}\}_{s \le t})\mid  \le K_{1} \sup_{s \le t} \norm{X_{s}-Y_{s}}
	\end{equation}
	
	\begin{equation}
		\int_0^r \mid a_{i}(t,\omega,\{X_{s}\}_{s \le t})-a_{i}(t,\omega,\{Y_{s}\}_{s \le t})\mid dt \leq K_{2} \int_0^r\norm{X_{t}-Y_{t}}dt
	\end{equation}
	
	\begin{equation}
		\mid b_{i,j}(t,\omega,\{X_{s}\}_{s \le t})-b_{i,j}(t,\omega,\{Y_{s}\}_{s \le t})\mid \le K_{3} \sup_{s \le t} \norm{X_{s}-Y_{s}}
	\end{equation}
	\begin{equation}
		\int_0^r \mid b_{i,j}(t,\omega,\{X_{s}\}_{s \le t})-b_{i,j}(t,\omega,\{Y_{s}\}_{s \le t})\mid^2 dt \leq K_{4} \int_0^r\norm{X_{t}-Y_{t}}^2dt
	\end{equation}
	with $K_{1},K_{2},K_{3},K_{4}$ being constants.

	$(ii)\,\Gamma:\mathbb{R}^n \times U \rightarrow  \mathbb{R}^n$ satisfy
	\begin{equation}
		\norm{\Gamma(x,u)}\le C \vee \norm{x} \, \text{and} \, 
	\end{equation}
	\begin{equation}\label{gam}
		\norm{\Gamma(x,u)-\Gamma(y,v)} \le \norm{(x,u)-(y,v)} \, \text{for all u,v}  \in U \text{and} \, x,y \in \mathbb{R}^n
	\end{equation}
\end{Ass}
\begin{Rem}
	These assumptions on $a$ and $b$ imply \\
	\begin{equation}\mid a_{i}(t,\omega,\{X_{s}\}_{s \le t})\mid^p  \le C (\sup_{s \le t} \norm{X_{s}}^p+1) 
	\end{equation}
	for any $p\ge 1$.
\end{Rem}
The following definition and theorem  are found in \cite{Pro}. 
\begin{Def}
	We say that $F:[0,T] \times \Omega \times\mathcal{D}^{n}\rightarrow \mathbb{R}$ is a functional Lipschitz operator if for any $X,Y \in (\mathcal{S}^{p})^{n}$ we have that
	
	(i) for $\tau \in \mathcal{T}_{0}$, if $X^{\tau-}=Y^{\tau-}$ then $F(t,\omega,X)^{\tau-}=F(t,\omega,Y)^{\tau-}$
	
	(ii) there is a finite increasing process $K_{t}$ such that
	\begin{equation}
		\mid F(t,\omega,X)-F(t,\omega,Y) \mid \le K_{t}\sup_{s\le t}\norm{X-Y}_{s}
	\end{equation}
	
\end{Def}
where $X_{t}^{\tau-}=X_{t}\chi_{\{t<\tau\}}+X_{\tau-}\chi_{\{t\ge\tau\}}.$
\begin{Thm}\label{P1}\cite{Pro}
	Let $F_{1},F_{2}$ be matrices with components that are functional Lipschitz operators. Then there is a unique function $X(t,\omega,x)$ on $\mathbb{R}_{+}\times\Omega\times\mathbb{R}^n$ such that
	
	(i)for each $x$, $X^{x}_{t}=X(t,\omega,x)$ is a solution of
	
	$X^{x}_{t}=x+\int_{0}^{t}F_{1}(X^{x})_{s}ds+\int_{0}^{t}F_{2}(X^{x})_{s}dB_{s} \quad (*)$
	
	(ii)for a.e. $\omega$, the flow $x\rightarrow X(\cdot,\omega,x)$ from $\mathbb{R}^n$ into $\mathcal{D}^n$ is continuous in the topology of uniform convergence on compacts.
\end{Thm} 
Our definition of a functional Lipschitz operator is less general than the one in \cite{Pro}. The definition we give, which the coefficients in Assumptions \ref{A1} falls within, is mentioned as the principle case in that reference.

With these assumptions, definitions and results at hand, we proceed by defining the dynamics for a given $\nu \in \mathcal{A}$. Let
\begin{equation}\label{ha1}
	dX^{x,\nu,0}_{t}=x+a(t,\omega,\{X^{x,\nu,0}_{s}\}_{s \le t})dt+b(t,\omega,\{X^{x,\nu,0}_{s}\}_{s \le t})dB_{t} \quad 0 \le t \le T 
\end{equation}
and recursively define
\begin{equation}\label{ha2}
	\begin{aligned}
		&dX^{x,\nu,j}_{t}=a(t,\omega,\{X^{x,\nu,j}_{s}\}_{s \le t})dt+b(t,\omega,\{X^{x,\nu,j}_{s}\}_{s \le t})dB_{t} \quad \tau_{j}< t \le T\\
		&X^{x,\nu,j}_{\tau_{j}}=\Gamma(X^{x,\nu,j-1}_{\tau_{j}},\xi_{j}) \\
		&X^{x,\nu,j}_{t}=X^{x,\nu,j-1}_{t} \quad 0 \le t  < \tau_{j}.
	\end{aligned}
\end{equation}
To obtain our controlled state we put $\limsup_{j\rightarrow \infty} X^{x,\nu,j}=X^{x,\nu}$.

The average performance of the control is measured by the following functional,

\begin{equation}
	J(\nu)=\mathbb{E}[\int_{0}^{T}f(s,X^{\nu}_{s})ds-\sum_{i} \ell(X^{\nu}_{\tau_{i}},\xi_{i},\tau_{i})]
\end{equation}
where $f$ and $\ell$ are subject to the following constraints.
\begin{Ass}
	$f:[0,T]\times \mathbb{R}^n \rightarrow \mathbb{R}$, $\ell: \mathbb{R}^n\times U \times [0,T]\rightarrow \mathbb{R}$,
	and $f(t,0)$ is continuous along with the following additional constraints
	\begin{equation}
		\begin{aligned}
			\mid f(t,x)-f(t,y) \mid &\le K_{5}\norm{x-y} \\
			\ell({x},u,t) &>K_{6}>0 \\
			\mid\ell(x_{1},u_{1},t_{1})-\ell(x_{2},u_{2},t_{2}) \mid &\le K_{7} \norm{(x_{1},u_{1},t_{1})-(x_{2},u_{2},t_{2})}
		\end{aligned}
	\end{equation}
	for constants $K_{5}, K_{6}$ and $K_{7}$.
\end{Ass}
\begin{Rem}
	These assumptions on $f$ implies 
	\begin{equation}
		\mid f(t,x) \mid^p \le C(1+ \norm{x}^p)
	\end{equation} 
	for any $p\ge 1$.
\end{Rem}

The problem of finding an optimal control can be stated as follows;
\begin{Pro}\label{vafan}
	Given $(a,b,\Gamma,f,\ell)$ find $\nu^{*}=(\tau^{*}_{i},\xi^{*}_{i})_{i\in \mathbb{N}} \in \mathcal{A}$ such that
	\begin{equation}
		J((\tau^{*}_{i},\xi^{*}_{i})_{i\in \mathbb{N}})=\sup_{\nu \in \mathcal{A}}J((\tau_{i},\xi_{i})_{i\in \mathbb{N}})
	\end{equation}
\end{Pro}

We proceed by stating a few results which we will need in order to show that this problem has a solution, the first of which is the most important.

We first recall the notion of a process being of class $[D]$.
\begin{Def}
	We say that a process $X_{t}$ is of class $[D]$ if $\{X_{\tau}:\tau < \infty\}$ is uniformly integrable.
\end{Def}

\begin{Thm}\label{Snell}[Snell envelope]\cite{Ka}

	Let $X_{t}$ be a process which is  $\mathbb{R}$-valued, adapted, cadlag and of class $[D]$. Then there exists a unique smallest dominating supermartingale $Z^{X}$ that is also $\mathbb{R}$-valued, adapted, cadlag and of class $[D]$. The process $Z^{X}$ is called the Snell envelope of $X$ and it has the following properties:

	(i)For any stopping time $\theta$ we have
	\begin{equation}
		Z^{X}_{\theta}=\esssup _{\tau \in \mathcal{T}_{\theta}}\mathbb{E}[X_{\tau} \mid \mathcal{F}_{\theta}] \,\, (\text{and then} \ Z^{X}_{T}=X_{T}).
	\end{equation}

	(ii)If $X$ is continuous, $\theta$ is a stopping time and we let  $\tau_{\theta}^{*}=\inf\{s\ge \theta : Z^{X}_{s}=X_{s} \} \wedge T$ then  $\tau_{\theta}^{*}$ is optimal after $\theta$   i.e
	\begin{equation}
		Z^{X}_{\theta}=\mathbb{E}[Z^{X}_{\tau_{\theta}^{*}} \mid \mathcal{F}_{\theta}]=\mathbb{E}[X_{\tau_{\theta}^{*}} \mid \mathcal{F}_{\theta}]=\esssup_{\tau \in \mathcal{T}_{\theta}}\mathbb{E}[X_{\tau} \mid \mathcal{F}_{\theta}]
	\end{equation}
	
	(iii)If $ (X^{n})_{n\ge 0}$ and  $X$ are cadlag of class [D] such that $ (X^{n})_{n\ge 0}$
	converges increasingly and pointwisely to $X$ then $Z^{X^{n}}$
	converges increasingly and pointwisely to $ Z^{X}$. Moreover if  $X$ is in $ \mathcal{S}^p_{c}$ then  $Z^{X}$ is in $ \mathcal{S}^p_{c}$.
	
\end{Thm}
\begin{Rem}
	The results of \cite{Ka} are obtained in a very general framework. A less general reference is \cite{KQ}.
\end{Rem}
The following definition  and theorems belong to the so called general theory of stochastic processes, proofs can be found in $\cite{KT}$. They will be needed in the construction of the optimal control.
\begin{Def}
	Given $A \subset \Omega\times\mathbb{R}^n$ we define the projection of $A$ onto $\Omega$ by $\pi_{\Omega}(A)=\{\omega\in\Omega :\exists x \in \mathbb{R}^{n} , (\omega,x)\in A\}$
\end{Def}

\begin{Thm}\label{mo}[Measurable projection]
	
	Let $(\Omega,\mathcal{F},\mathbb{P})$ be complete. For every $A \in \mathcal{F}\otimes\mathcal{B}(\mathbb{R}^n)$ the set $\pi_{\Omega}(A)$ is $\mathcal{F}$-measurable. 
\end{Thm}
\begin{Cor}
	Let $(\Omega,\mathcal{F},\mathbb{P})$ be complete and $h(\omega,x)$ be a real-valued, measurable function on $(\Omega \times \mathbb{R}^n,\mathcal{F}\otimes \mathcal{B}(\mathbb{R}^n))$. Then given any $A \in \mathcal{F} \otimes \mathcal{B}(\mathbb{R}^n)$
	\begin{equation}
		g(\omega):=\sup_{x \in \mathbb{R}^n}\{h(\omega,x):(\omega,x)\in A \}
	\end{equation}
	is $\mathcal{F}$-measurable.
	\begin{proof}
		Given any real constant $K$ the following holds $\{ g(\omega)> K \}=\pi_{\Omega}(A\cap h^{-1}((K,\infty]))$. As $h$ is measurable $A\cap h^{-1}((K,\infty])\in \mathcal{F} \otimes \mathcal{B}(\mathbb{R}^n)$, applying Theorem \ref{mo}  finishes the proof.
	\end{proof}
\end{Cor}
\begin{Thm}\label{kokoko}[Measurable selection]
	
	Let $(\Omega,\mathcal{F},\mathbb{P})$ be complete. For any $A \in \mathcal{F}\otimes\mathcal{B}(\mathbb{R}^n)$ there is a $\mathcal{F}$-measurable function $\beta$ taking values in $\mathbb{R}^n \cup \{\infty\}$ such that
	\begin{equation}
		\{ (\omega,\beta(\omega))\in \Omega \times \mathbb{R}^n  \} \subset A \qquad and \qquad \{\omega \in \Omega: \beta(\omega)\in \mathbb{R}^n   \}=\pi_{\Omega}(A)
	\end{equation}
\end{Thm}
\begin{Cor}\label{corsel}
	Let $(\Omega,\mathcal{F},\mathbb{P})$ be complete and let $h(\omega,x)$ be a measurable function on $(\Omega \times \mathbb{R}^n,\mathcal{F}\otimes \mathcal{B}(\mathbb{R}^n))$, such that for a.e. $\omega$ the map $x \rightarrow h(\omega,x)$ is upper semi-continuous. Then given $U \subset \mathbb{R}^n$ compact, there exists a $\mathcal{F}$-measurable function $\beta$ such that
	\begin{equation}
		h(\omega,\beta(\omega))=\sup_{x \in \mathbb{R}^n}\{h(\omega,x):(\omega,x)\in \Omega \times U \}
	\end{equation}
	a.s.
	\begin{proof}
		As $A:=\Omega \times U \in \mathcal{F} \otimes \mathcal{B}(\mathbb{R}^{n}), \, g(\omega)=\sup_{x \in \mathbb{R}^n} \{h(\omega,x):(\omega,x)\in A\}$ is $\mathcal{F}$-measurable, $h$ is $\mathcal{F} \otimes \mathcal{B}(\mathbb{R}^{n})$-measurable and $B:=\{(\omega,x)\in \Omega\times U:h(\omega,x)=g(\omega)\}\in\mathcal{F} \otimes \mathcal{B}(\mathbb{R}^{n})$.
		
		Hence by Theorem \ref{kokoko} there is a $\mathcal{F}$-measurable function $\beta$ such that $\{(\omega,\beta(\omega))\in \Omega \times \mathbb{R}^{n}\} \subset B$ and $\{\omega : \beta(\omega)\in \mathbb{R}^{n}\}=\pi_{\Omega}(B)$. Thus, since $U$ is compact and $b \rightarrow h(\omega,b)$ is u.s.c. on $\Omega \setminus \mathcal{N}$ for a nullset $\mathcal{N}$, we get $B^{\omega}:=\{ b \in U : (\omega,b) \in B \}=\{ b \in U : h(\omega,b)=g(\omega) \} \ne \emptyset$ for all $\omega \in \Omega \setminus \mathcal{N}$, hence $\mathbb{P}(\pi_{\Omega}(B))=1$.
	\end{proof}
\end{Cor}

\section{Lyaponov-type stability of the state dynamic in connection to impulses}
In this section we will prove an essential property of the state dynamics. In particular, using the flow property of SFDEs, we show that the solution corresponding to the control $\nu_{1} \circ (t,u)\circ \nu_{2}$
converges to the solution corresponding to $\nu_{1} \circ (\hat{t},\hat{u})\circ \nu_{2}$ as $(t,u) \rightarrow (\hat{t},\hat{u})$. We start by stating the following well known lemma.
\begin{Lem}
	Let $p\in [1,\infty)$ and $x \in \mathbb{R}^{n}$, $F$ be functional Lipschitz and $\sup_{t}K_{t}\le k$ a.s. Then the solution $(*)$ satisfy
	\begin{equation}
		\mathbb{E}[ \sup_{t\in [0,T]} \norm{ X^{x}_{t}} ^p] \le C(1+ \norm{x}^p)
	\end{equation}
\end{Lem}

\begin{Thm}\label{dyn}
	Under Assumption 1, our controlled SFDE defined via \ref{ha1} and \ref{ha2}, admits a unique solution for any $\nu \in \mathcal{A}$. Moreover, we have that
	\begin{equation}
		\sup_{\nu \in \mathcal{A}_{f}}\mathbb{E}[\sup_{t\in [0,T]}\norm{X^{x,\nu}_{t}}^{q}]<C_{T,q}
	\end{equation}for any $q\ge 1 $ and
	\begin{equation}\label{nora}
		\sup_{\nu_2 \in \mathcal{A}^{k}_{f}}\mathbb{E}[\sup_{s\in [\hat{t},T]}\norm{ X^{x,\nu_{1} \circ (t,u)\circ \nu_{2} }_{s}-X^{x,\nu_{1} \circ (\hat{t},\hat{u})\circ \nu_{2}}_{s}}^{4+2m}] \le C \norm{(t-\hat{t},u-\hat{u})}^{2+m}
	\end{equation}
	for $\nu_{1}\in \mathcal{A}_{f}$, where $m$ is the dimension of the control space $U$.

\begin{proof}

$X^{x,\nu,j}$ exists uniquely for each $j$ by Theorem \ref{P1} and thus so does the $\limsup$ as $\tau_{j} \rightarrow T$. 

For the second statement we note that $X^{x,\nu,j}=X^{x,\nu,j-1}$ on $[0,\tau_{j})$ and 
\begin{equation*}
	X^{x,\nu,j}_{t}=\Gamma(X^{x,\nu,j-1}_{\tau_{j}},\xi_{j})+\int_{\tau_{j}}^{t} a(s,\omega,\{X^{x,\nu,j}_{r}\}_{r \le s})ds+\int_{\tau_{j}}^{t}b(s,\omega,\{X^{x,\nu,j}_{r}\}_{r \le s})dB_{s},
\end{equation*}
on $[\tau_{j},T]$. Using Assumption \ref{A1} (ii) repeatedly we get,
\begin{equation}
	\begin{aligned}
		&\norm{ X^{x,\nu,j}_{t} }^{2} \le \norm{ X^{x,\nu,j}_{\tau_{j}} }^{2} + 2\int_{\tau_{j}}^{t}X^{x,\nu,j}_{s}dX^{x,\nu,j}_{s}+ \int_{\tau_{j}}^{t}d[X^{x,\nu,j},X^{x,\nu,j}]_{s}\\
		&\le C\, \vee \norm{ X^{x,\nu,j-1}_{\tau_{j}} }^{2} +   2\int_{\tau_{j}}^{t}X^{x,\nu,j}_{s}dX^{x,\nu,j}_{s}+ \int_{\tau_{j}}^{t}d[X^{x,\nu,j},X^{x,\nu,j}]_{s}.\\
	\end{aligned}
\end{equation}
Since if $\norm{ X^{x,\nu,j}_{t} }^{2}> C$ and $\norm{ X^{x,\nu,j}_{s} }^{2}\le C$ for some $s\in [0,t)$ then there is a largest $\hat{s}<t$ such that $\norm{ X^{x,\nu,j}_{\hat{s} }}^{2}\le C$, we know that there are no interventions that increase the magnitude of $\norm{ X^{x,\nu,j}_{s} }$ on $(\hat{s},t]$. Hence letting $\hat{s}=\sup\{s\ge0:X^{x,\nu,j}_{s} \le C\} \vee 0$, we use induction to obtain,
\begin{equation}
	\begin{aligned}
		&\le C \,  + \Big(2\sum_{i} \int_{\tau_{i-1}\vee \hat{s}}^{\tau_{i}\wedge t}X^{x,\nu,i}_{s}dX^{x,\nu,i}_{s}+\sum_{i} \int_{\tau_{i-1}\vee \hat{s}}^{\tau_{i}\wedge t}d[X^{x,\nu,i},X^{x,\nu,i}]_{s}\Big)\\
		&+2\int_{\tau_{j}\vee \hat{s}}^{t}X^{x,\nu,j}_{s}dX^{x,\nu,j}_{s}+ \int_{\tau_{j}\vee \hat{s}}^{t}d[X^{x,\nu,j},X^{x,\nu,j}]_{s}
	\end{aligned}
\end{equation}
By raising this by a power of $\frac{q}{2}$ for $q\ge 2$, taking supremum followed by expectation and then using Burkholder-Davis-Gundy inequality as well as some elementary estimates we get
\newpage
\begin{equation}
	\begin{aligned}
		\mathbb{E}\Big[\sup_{t\in [0,s] } \norm{X^{x,\nu,j}_{t} }^{q}\Big] \le C + \int_{0}^{t}\mathbb{E}\Big[\sup_{r \in [0,s]} \norm{ X^{x,\nu,j}_{r} }^{q}\Big] ds.
	\end{aligned}
\end{equation}
Since $\nu$ is arbitrary and any estimate along the way is independent of $\nu$, the statement follows from Grönvall's inequality and then taking the limit in $j$. To obtain the bound for $q \in [1,2)$ one simply applies Jensens inequality. 

For the last statement, assume that $t < \hat{t}$ and consider,
\begin{equation*}
	\begin{aligned}
		&\mathbb{E}\sup_{\hat{t} \le s\le r}\norm{ X^{x,\nu_{1} \circ (t,u) \circ \nu_{2}}_{s}-X^{x,\nu_{1} \circ (\hat{t},\hat{u}) \circ \nu_{2}}_{s}}^{p}\\
		&\le C\Big[\mathbb{E}\norm{X^{x,\nu_{1} \circ (t,u) \circ \nu_{2}}_{\hat{t}}-X^{x,\nu_{1} \circ (\hat{t},\hat{u}) \circ \nu_{2}}_{\hat{t}}}^{p}\\
		&+\mathbb{E}\Big(\int_{\hat{t}}^{r}\mid \mid a(v,\omega,\{X^{x,\nu_{1} \circ (t,u) \circ \nu_{2}}_{z}\}_{z \le v})-a(v,\omega,\{X^{x,\nu_{1} \circ (\hat{t},\hat{u}) \circ \nu_{2}}_{z}\}_{z \le v})\mid \mid dv\Big)^{p}\\
		&+\mathbb{E}\sup_{\hat{t}\le s\le r}\mid \mid\int_{\hat{t}}^{s} b(v,\omega,\{X^{x,\nu_{1} \circ (t,u) \circ \nu_{2}}_{z}\}_{z \le v})-b(v,\omega,\{X^{x,\nu_{1} \circ (\hat{t},\hat{u}) \circ \nu_{2}}_{z}\}_{z \le v})dB_{v}\mid\mid^{p}\Big]
\end{aligned}
\end{equation*}
		We start by estimating the first term, which could potentially be large due to $v_{2}$ containing times less than or equal to $\hat{t}$.
\begin{equation*}
	\begin{aligned}
		&\mathbb{E}\norm{X^{x,\nu_{1} \circ (t,u) \circ \nu_{2}}_{\hat{t}}-X^{x,\nu_{1} \circ (\hat{t},\hat{u}) \circ \nu_{2}}_{\hat{t}}}^{p}\\
		&\leq\mathbb{E} \norm{ \Gamma(X^{x,\nu_1\circ(t,u)\circ\nu_2,N_1+N_2-1}_{\tau_{N_2}},\xi_{N_2})-\Gamma(X^{x,\nu_1\circ(\hat t,\hat u)\circ\nu_2,N_1+N_2-1}_{\hat t},\xi_{N_2})}^{p}\\
		&+\mathbb{E}\norm{X^{x,\nu_1\circ(t,u)\circ\nu_2,N_1+l}_{\hat t}-X^{x,\nu_1\circ(t,u)\circ\nu_2,N_1+N_2}_{\tau_{N_2}}}^{p}.\\
\end{aligned}
\end{equation*}
Using Lipschitz condition on $\Gamma$ and repeating this we get,
\begin{equation*}
\begin{aligned}
		&\mathbb{E}\norm{X^{x,\nu_{1} \circ (t,u) \circ \nu_{2}}_{\hat t}-X^{x,\nu_{1} \circ (\hat{t},\hat{u}) \circ \nu_{2}}_{\hat{t}}}^{p}\\
		&\leq\mathbb{E} \norm{\Gamma(X^{x,\nu_{1} \circ (t,u) \circ \nu_{2},N_1-1}_{t},u)-\Gamma(X^{x,\nu_{1} \circ (\hat{t},\hat{u}) \circ \nu_{2},N_1-1}_{\hat{t}},\hat u)}^{p}\\
		&+\sum_{j=0}^{N_{2}-1} \mathbb{E}\norm{X^{x,\nu_1\circ(t,u)\circ\nu_2,N_1+j}_{\tau_{j+1}}-X^{x,\nu_1\circ(t,u)\circ\nu_2,N_1+j}_{\tau_j}}^{p}
	\end{aligned}
\end{equation*}
Exploiting the flow property of the state we obtain,
\begin{equation*}
	\begin{aligned}
		&\mathbb{E}\norm{X^{x,\nu_{1} \circ (t,u) \circ \nu_{2}}_{\hat t}-X^{x,\nu_{1} \circ (\hat{t},\hat{u}) \circ \nu_{2}}_{\hat{t}}}^{p}
		\leq \mathbb{E}\norm{(X^{x,\nu_{1} \circ (t,u) \circ \nu_{2}}_{t}-X^{x,\nu_{1} \circ (\hat{t},\hat{u}) \circ \nu_{2}}_{\hat{t}},u-\hat u)}^{p}\\
		&+\sum_{j=0}^{N_{2}-1}\mathbb{E} \mid\mid\int_{\tau_{j}}^{\tau_{j+1}} a(v,\omega,\{X^{x,\nu_{1}\circ (\hat{t},\hat{u})\circ \nu_{2}}_{z}\}_{z \le v})dv\\
		&+\int_{\tau_{j}}^{\tau_{j+1}}  b(v,\omega,\{X^{x,\nu_{1}\circ (\hat{t},\hat{u})\circ \nu_{2}}\}_{z \le v})dB_{v}\mid\mid^{p}\\
		&\leq\mathbb{E}\mid\mid(\int_{t}^{\hat t} a(v,\omega,\{X^{x,\nu_{1}\circ (\hat{t},\hat{u})\circ \nu_{2}}_{z}\}_{z \le v})dv\\
		&+\int_{t}^{\hat t}  b(v,\omega,\{X^{x,\nu_{1}\circ (\hat{t},\hat{u})\circ \nu_{2}}\}_{z \le v})dB_{v},u-\hat u)\mid\mid^{p} \\
		&+\sum_{j=0}^{N_{2}-1}C\mathbb{E}\Big( \int_{t}^{\hat t}\norm{ a(v,\omega,\{X^{x,\nu_{1}\circ (\hat{t},\hat{u})\circ \nu_{2}}_{z}\}_{z \le v})}dv\Big)^{p}\\
		&+\sum_{j=0}^{N_{2}-1}C\mathbb{E}\norm{\int_{\tau_{j}}^{\tau_{j+1}}  b(v,\omega,\{X^{x,\nu_{1}\circ (\hat{t},\hat{u})\circ \nu_{2}}\}_{z \le v})dB_{v}}^{p}\\
		&\leq\mathbb{E}\mid\mid(\int_{t}^{\hat t} a(v,\omega,\{X^{x,\nu_{1}\circ (\hat{t},\hat{u})\circ \nu_{2}}_{z}\}_{z \le v})dv\\
		&+\int_{t}^{\hat t}  b(v,\omega,\{X^{x,\nu_{1}\circ (\hat{t},\hat{u})\circ \nu_{2}}\}_{z \le v})dB_{v},u-\hat u)\mid\mid^{p} \\
		&+C\sum_{j=0}^{N_{2}-1}\mathbb{E}\Big( (\hat t -t)^{p-1}\int_{t}^{\hat t}\norm{ a(v,\omega,\{X^{x,\nu_{1}\circ (\hat{t},\hat{u})\circ \nu_{2}}_{z}\}_{z \le v})}^{p}dv\Big)\\
		&+C\sum_{j=0}^{N_{2}-1}\mathbb{E}\sup_{\tau_{j}\le s\le \tau_{j+1}}\norm{\int_{\tau_{j}}^{s}  b(v,\omega,\{X^{x,\nu_{1}\circ (\hat{t},\hat{u})\circ \nu_{2}}\}_{z \le v})dB_{v}}^{p}
\end{aligned}
\end{equation*}
Using Burkholder-Davis-Grundy and the assumptions on the coefficients we get
\begin{equation*}
\begin{aligned}
		&\le\mathbb{E}\Big(C(\int_{t}^{\hat t} a(v,\omega,\{X^{x,\nu_{1}\circ (\hat{t},\hat{u})\circ \nu_{2}}_{z}\}_{z \le v})dv)^2\\
		&+C(\int_{t}^{\hat t}  b(v,\omega,\{X^{x,\nu_{1}\circ (\hat{t},\hat{u})\circ \nu_{2}}\}_{z \le v})dB_{v})^2+(u-\hat u)^2\Big)^\frac{p}{2}\\
		&+C\sum_{j=0}^{N_{2}-1}(\hat{t}-t)^{p}\mathbb{E}C(\sup_{t}\norm{X^{x,\nu_{1} \circ (\hat t,\hat u) \circ \nu_{2}}_{t}}^{p}+1)\\
		&+C\sum_{j=0}^{N_{2}-1}\mathbb{E}C_{p}\norm{\int_{t}^{\hat t}  b(v,\omega,\{X^{x,\nu_{1}\circ (\hat{t},\hat{u})\circ \nu_{2}}\}_{z \le v})^2dv}^\frac{p}{2}\\
	\end{aligned}
\end{equation*}

\begin{equation*}
	\begin{aligned}
		&\le\mathbb{E}\Big(\hat{C}(\int_{t}^{\hat t} a(v,\omega,\{X^{x,\nu_{1}\circ (\hat{t},\hat{u})\circ \nu_{2}}_{z}\}_{z \le v})dv)^{p}\\
		&+\hat{C}(\int_{t}^{\hat t}  b(v,\omega,\{X^{x,\nu_{1}\circ (\hat{t},\hat{u})\circ \nu_{2}}\}_{z \le v})dB_{v})^{p}\\
		&+\hat{C}(u-\hat u)^{p}\Big)+C^2\sum_{j=0}^{N_{2}-1}(\hat{t}-t)^{p}\mathbb{E}(\sup_{t}\norm{X^{x,\nu_{1} \circ (\hat t,\hat u) \circ \nu_{2}}_{t}}^{p}+1)\\
		&+CC_{p}\sum_{j=0}^{N_{2}-1}\mathbb{E}(\hat t-t)^{\frac{p}{2}-1}\int_{t}^{\hat t}  b(v,\omega,\{X^{x,\nu_{1}\circ (\hat{t},\hat{u})\circ \nu_{2}}\}_{z \le v})^{p}dv\\
		&\le C(CN_2-1+\hat C)(\hat{t}-t)^{p}\mathbb{E}(\sup_{t}\norm{X^{x,\nu_{1} \circ (\hat t,\hat u) \circ \nu_{2}}_{t}}^{p}+1)\\
		&+C(CC_{p}N_2-1+\hat C)(\hat{t}-t)^\frac{p}{2}\mathbb{E}(\sup_{t}\norm{X^{x,\nu_{1} \circ (\hat t,\hat u) \circ \nu_{2}}_{t}}^{p}+1)+\hat{C}(u-\hat u)^{p}\\
		&\le A(\hat{t}-t)^\frac{p}{2}+B(u-\hat{u})^\frac{p}{2}\le \max{\{A,B\}}\norm{(t-\hat{t},u-\hat{u})}^\frac{p}{2},\\
\end{aligned}
\end{equation*}
where $A$ and $B$ are finite due to the previous statement. Moving on to the remaining terms we have
\begin{equation*}
	\begin{aligned}
		&+\mathbb{E}\Big(\int_{\hat{t}}^{r}\mid \mid a(v,\omega,\{X^{x,\nu_{1} \circ (t,u) \circ \nu_{2}}_{z}\}_{z \le v})-a(v,\omega,\{X^{x,\nu_{1} \circ (\hat{t},\hat{u}) \circ \nu_{2}}_{z}\}_{z \le v})\mid \mid dv\Big)^{p}\\
		&+\mathbb{E}\sup_{\hat{t}\le s\le r}\mid \mid\int_{\hat{t}}^{s} b(v,\omega,\{X^{x,\nu_{1} \circ (t,u) \circ \nu_{2}}_{z}\}_{z \le v})-b(v,\omega,\{X^{x,\nu_{1} \circ (\hat{t},\hat{u}) \circ \nu_{2}}_{z}\}_{z \le v})dB_{v}\mid\mid^{p}\\
		&\le\mathbb{E}\Big(K_{2}\int_{0}^{r} \norm{X^{x,\nu_{1} \circ (t,u) \circ \nu_{2}}_{v}-X^{x,\nu_{1} \circ (\hat{t},\hat{u}) \circ \nu_{2}}_{v}} dv\Big)^{p}\\
		&+C_{p}\mathbb{E}\norm{\int_{\hat{t}}^{r} (b(v,\omega,\{X^{x,\nu_{1} \circ (t,u) \circ \nu_{2}}_{z}\}_{z \le v})-b(v,\omega,\{X^{x,\nu_{1} \circ (\hat{t},\hat{u}) \circ \nu_{2}}_{z}\}_{z \le v}))^2dv}^\frac{p}{2}\\
		&\le\mathbb{E}\cdot r^{p-1} \cdot K_{2}^2 \int_{0}^{r}  \norm{ X^{x,\nu_{1} \circ (t,u) \circ \nu_{2}}_{v}- X^{x,\nu_{1} \circ (\hat{t},\hat{u}) \circ \nu_{2}}_{v}}^{p}dv\\
		&+C_{p}\cdot K_{4}\mathbb{E}\int_{0}^{r} \norm{ X^{x,\nu_{1} \circ (t,u) \circ \nu_{2}}_{v}- X^{x,\nu_{1} \circ (\hat{t},\hat{u}) \circ \nu_{2}}_{v}}^{p} dv \\
		&\le(r^{p-1}\cdot K_{2}^2+C_{p} \cdot K_{4})\int_{0}^{r}\mathbb{E}\sup_{\hat{t}\le s \le v} \norm{ X^{x,\nu_{1} \circ (t,u) \circ \nu_{2}}_{v}- X^{x,\nu_{1} \circ (\hat{t},\hat{u}) \circ \nu_{2}}_{v}}^{p} dv.
\end{aligned}
\end{equation*}
Which we justify by similar reasoning as above. Letting $p=4+2m$, the latter estimate allows us to use Grönwall's lemma and the former give us the second,
\begin{equation*}
	\begin{aligned}
		&\mathbb{E}\sup_{s\le r}\norm{ X^{x,\nu_{1} \circ (t,u) \circ \nu_{2}}_{s}-X^{x,\nu_{1} \circ (\hat{t},\hat{u}) \circ \nu_{2}}_{s}}^{4+2m}\\
		&\le \mathbb{E} \norm{X^{x,\nu_{1} \circ (t,u)\circ\nu_{2}}_{\hat{t}}-X^{x,\nu_{1} \circ (\hat{t},\hat{u}) \circ \nu_{2}}_{\hat{t}}}^{4+2m}\cdot e^{CT}\\
		&\le C \norm{(t-\hat{t},u-\hat{u})}^{2+m}
	\end{aligned}
\end{equation*}

\end{proof}
\end{Thm}
\section{Verification theorem}
In this section we present the main result of the paper which is the characterisation of an optimal control to Problem \ref{vafan}. Theorems of this kind are in general known as verification theorems. Such theorems are usually based on a large set of assumptions, often including the existence of a solution to a certain equation and some additional hypotheses regarding existence of an optimal control. 

In contrast to the general concept of a verification theorem, our theorem, merely assumes the existence of solutions to a certain family of equations and as a direct consequence we also obtain existence of an optimal control. We pay a price however,  this family also needs to be interrelated.

We will follow the approach taken in $\cite{DJ}$ which is roughly the following.

By assuming existence of a verification family, which is a family of interconntected Snell envelopes, we will be able to recreate the performance functional by using Theorem $\ref{Snell}$ iteratively. Due to Theorem $\ref{Snell}$(ii) this scheme will also provide us with the optimal control.

To prove that such a family exists, we will, as in $\cite{DJ}$, define a sequence of verification families and prove that the limit exhibits the required properties. The reason for proving the verification theorem ahead of the existence is that we will use a constrained version of it in order to prove the latter.

We start with the definition of a verification family.
\begin{Def}\label{lo}
	We say that a family of continuous supermartingales $\{Y^{\nu} \}_{\nu \in \mathcal{A}_{f}}$ is a verification family if it satisfies:
	
	\begin{equation}
		\begin{aligned}
			&(i)\,Y_{s}^{\nu}=\esssup_{\tau \in \mathcal{T}_{s}} \mathbb{E}[\int_{s}^{\tau}f(t,X^{\nu}_{t})dt+\sup_{u\in U}\{Y^{\nu\circ (\tau,u)}_{\tau}- \ell(X^{\nu}_{\tau},u,\tau)\}\mid \mathcal{F}_{s}]\\
			&(ii)\,\sup_{u\in U}\{Y^{\nu\circ (s,u)}_{s}-\ell(X^{\nu}_{s},u,s)\} \text{ is a.s. continuous and adapted in $s$}.\\
			&(iii)\, Y^{\nu\circ(\tau,u)}_\tau-\ell(X^\nu_\tau,u,\tau)\text{ is upper semi-continous in $u$ for all $\tau$}
		\end{aligned}
	\end{equation}
	
\end{Def}
The following Lemma, from $\cite{DJ}$, simplifies the proof of the verification theorem.
\begin{Lem}\label{tt}
	The supremum in Problem \ref{vafan} over $\mathcal{A}$ and $\mathcal{A}_{f}$  coincide.
	\begin{proof}
		In the spirit of $\cite{DJ}$ we let $(\tau_{i}$, $\xi_{i})_{i\in \mathbb{N}}\in \mathcal{A} \setminus \mathcal{A}_{f}$ and consider the set $B=\{\omega ; \tau_{i}<T \ \text{for all}\ i\}$. Since $(\tau_{i}$, $\xi_{i})_{i\in \mathbb{N}} \in \mathcal{A} \setminus \mathcal{A}_{f}$, we have $\mathbb{P}(B)>0$. Hence,
		\begin{equation}
			\begin{aligned}
				&J((\tau_{i}, \xi_{i})_{i\in \mathbb{N}})\le \mathbb{E}[\int_{0}^{T}\sup_{\hat{\nu}}f(s,X^{\hat{\nu}}_{s})ds]\\
				&-\mathbb{E}\Big[(\sum_{i} \ell(X^{\nu}_{\tau_{i}},\xi_{i},\tau_{i}))\chi_{B}\\
				&+(\sum_{i} \ell(X^{\nu}_{\tau_{i}},\xi_{i},\tau_{i}))\chi_{\Omega \setminus B}\Big]=-\infty,
			\end{aligned}
		\end{equation}
		due to Theorem $\ref{dyn}$, Assumption $2$ and  $\ell\ge K_{6}>0$.
	\end{proof}
	
\end{Lem}
We are now ready to state and prove our main result.
\begin{Thm}
	Suppose there exists a verification family. Then it satisfies
	
	\begin{equation}
		Y_{0}=\sup_{u\in \mathcal{A}}J(u),
	\end{equation}
	is unique and defines a solution to Problem \ref{vafan} via the control
	\begin{equation}\label{nko}
		\begin{aligned}
			\tau_{0}^{*}&:=0  \\
			\tau_{j}^{*}&:=\inf\Big\{s \ge \tau_{j-1}^{*}:Y^{\tau_{1}^{*},\ldots,\tau_{j-1}^{*}:\xi_{1}^{*},\ldots,\xi_{j-1}^{*}}_{s}=\\
			&\sup_{u\in U}\{ Y^{\tau_{1}^{*},\ldots,\tau_{j-1}^{*},s:\xi_{1}^{*},\ldots,\xi_{j-1}^{*},u}_{s}-\ell(X^{\tau_{1}^{*},\ldots,\tau_{j-1}^{*}:\xi_{1}^{*},\ldots,\xi_{j-1}^{*}}_{s},u,s)\} \Big\} \wedge T
		\end{aligned}
	\end{equation}
	where $\xi_{j}^{*}$ is a measurable selection of
	\begin{equation}\label{nkj}
		\xi_{j}^{*} \in \argmax_{\{u\in U\}} \{Y^{\tau_{1}^{*},\ldots,\tau_{i-1}^{*},\tau_{j}^{*}:\xi_{1}^{*},\ldots,\xi_{j-1}^{*},u}_{\tau_{j}^{*}}-\ell(X^{\tau_{1}^{*},\ldots,\tau_{i-1}^{*}:\xi_{1}^{*},\ldots,\xi_{j-1}^{*}}_{\tau_{j}^{*}},u,\tau_{j}^{*})\}
	\end{equation}
	\begin{proof}
		We start by noting that the recursion $(i)$ in $(4.1)$ also hold if we replace $s$ by a stopping time and that the supremum is attained. This follows from Definition \ref{lo} (i), Theorem $\ref{Snell}$(ii) and Corollary $\ref{corsel}$. Hence, for some one-step optimal control of one impulse, $(\tau^{*},\xi^{*})$, we have
		\begin{equation}
			Y_{\theta}^{\nu}= \mathbb{E}[\int_{\theta}^{\tau^{*}}f(t,X^{\nu}_{t})dt+Y^{\nu\circ (\tau^{*},\xi^{*})}_{\tau^{*}}- \ell(X^{\nu}_{\tau^{*}},\xi^{*},\tau^{*})\mid \mathcal{F}_{\theta}].
		\end{equation}
		Since $\nu$ was arbitrary we have, starting at $0$,
		\begin{equation}
			\begin{aligned}
				Y_{0}&=\mathbb{E}[\int_{0}^{\tau^{*}_{1}}f(t,X_{t})dt+Y^{(\tau^{*}_{1},\xi^{*}_{1})}_{\tau^{*}_{1}}- \ell(X_{\tau^{*}_{1}},\xi^{*}_{1},\tau^{*}_{1})].
			\end{aligned}
		\end{equation}
		Moreover, for any $j$ we have
		\begin{equation}
			\begin{aligned}
				&Y^{(\tau^{*}_{1},\xi^{*}_{1})\circ \ldots \circ (\tau^{*}_{j},\xi^{*}_{j})}_{\tau^{*}_{j}}\\
				&=\mathbb{E}[\int_{\tau^{*}_{j}}^{\tau^{*}_{j+1}}f(t,X^{(\tau^{*}_{1},\xi^{*}_{1})\circ \ldots \circ (\tau^{*}_{j},\xi^{*}_{j})}_{t})dt+Y^{(\tau^{*}_{1},\xi^{*}_{1})\circ \ldots \circ (\tau^{*}_{j+1},\xi^{*}_{j+1})}_{\tau^{*}_{j+1}}\\
				&-\ell(X^{(\tau^{*}_{1},\xi^{*}_{1})\circ \ldots \circ (\tau^{*}_{j},\xi^{*}_{j})}_{\tau^{*}_{j+1}},\xi^{*}_{j+1},\tau^{*}_{j+1})\mid \mathcal{F}_{\tau_{j}^*}].
			\end{aligned}
		\end{equation}
		Hence simply by inserting the latter into the former we obtain,
		\begin{equation}
			\begin{aligned}
				Y_{0}=&\mathbb{E}[\int_{0}^{\tau^{*}_{N}}f(t,X^{(\tau^{*}_{1},\xi^{*}_{1})\circ \ldots \circ (\tau^{*}_{N-1},\xi^{*}_{N-1})}_{t})dt-\sum_{i=1}^{N-1} \ell(X^{(\tau^{*}_{1},\xi^{*}_{1})\circ \ldots \circ (\tau^{*}_{i},\xi^{*}_{i})}_{\tau^{*}_{i}},\xi^{*}_{i},\tau^{*}_{i})\\
				&+Y^{(\tau^{*}_{1},\xi^{*}_{1})\circ \ldots \circ (\tau^{*}_{N},\xi^{*}_{N})}_{\tau^*_{N}}-\ell(X^{(\tau^{*}_{1},\xi^{*}_{1})\circ \ldots \circ (\tau^{*}_{N-1},\xi^{*}_{N-1})}_{\tau^{*}_{N}},\xi^{*}_{N},\tau^{*}_{N})],\\
			\end{aligned}
		\end{equation}
		for $\tau^*_{i}$ and $\xi^*_{i}$ defined as above due to $\{\tau^*_{i+1} <T \} \subset \{\tau^*_{i}<T\}$. Furthermore, this strategy must be finite. Assuming it is not, we can contradict the continuity of $Y$ using the same argument as in Lemma \ref{tt}. Thus by taking the limit we obtain $Y_{0}=J(\nu^{*})$.
		
		To complete the proof it remains to show that this strategy dominates any other $\nu \in \mathcal{A}_{f}$. This is seen by repeating the above argument taking into account the optimality characterisation in Theorem $\ref{Snell}$(ii).
	\end{proof}
\end{Thm}
\section{Existence of verification family}
In this section we consider the existence of the verification family from Section $4$. Hence, the main task is to prove the following theorem,
\begin{Thm}\label{Thm:neat}
	Under Assumptions $1\,\&\,2$, there exists a family of continuous supermartingales satisfying
	\begin{equation*}
		\begin{aligned}
			&(i)\,Y_{s}^{\nu}=\esssup_{\tau \in \mathcal{T}_{s}} \mathbb{E}[\int_{s}^{\tau}f(t,X^{\nu}_{t})dt+sup_{u\in U}\{Y^{\nu\circ (\tau,u)}_{\tau}- \ell(X^{\nu}_{\tau},u,\tau)\}\chi_{\{\tau<T\}}\mid \mathcal{F}_{s}]\\
			&(ii)\,\sup_{u\in U}\{Y^{\nu\circ (s,u)}_{s}-\ell(X^{\nu}_{s},u,s)\} \text{ is a.s continuous and adapted in $s$}.\\
			&(iii)\, Y^{\nu\circ(\tau,u)}_\tau-\ell(X^\nu_\tau,u,\tau)\text{ is upper semi-continous in $u$ for all $\tau$}
		\end{aligned}
	\end{equation*}

\end{Thm}
To prove this theorem we will use an approximating scheme similar to $\cite{DJ}$, where they allow the system to be intervened on $k$ times. In particular, using induction we will define the following families of processes
\begin{equation}
	\{Y^{\nu,0}_{s}\}_{\nu \in \mathcal{A}_{f} }=\{\mathbb{E}[\int_{s}^{T}f(t,X^{\nu}_{t})dt \mid \mathcal{F}_{s}]\}_{\nu \in \mathcal{A}_{f} }
\end{equation}
\begin{equation}\label{ohno}
	\begin{aligned}
		\{Y^{\nu,k}_{s}\}_{\nu \in \mathcal{A}_{f}}&=\{\esssup_{\tau \in \mathcal{T}_{s}}\mathbb{E}[\int_{s}^{\tau}f(t,X^{\nu}_{t})dt+\sup_{u\in U} \{Y^{\nu\circ (\tau,u),k-1}_{\tau}\\
		&-\ell(X^{\nu}_{\tau},u,\tau)\}\chi_{\{\tau < T\}} \mid  \mathcal{F}_{s}]\}_{\nu \in \mathcal{A}_{f}. }
	\end{aligned}
\end{equation}
The existence of such families is non-trivial, since it is not clear if the process inside of the expectation fulfils the conditions of Theorem $\ref{Snell}$. In order to prove that these families exist and that their limit is a verification family, we state and prove a few lemmas. 

\begin{Lem}\label{jcts}
	$Y^{\nu\circ (t,u),k}_{t}$  is a.s. continuous as a function of $(t,u)$ for any $k$.
\end{Lem}
\emph{Proof.} We start by proving the following representation
\begin{equation}\label{verf}
	\begin{aligned}
		&Y^{\nu\circ (t,u),k}_{t} \\
		&=\esssup_{\tau \in \mathcal{T}_{t}}\mathbb{E}[\int_{t}^{\tau}f(s,X^{\nu\circ (t,u)}_{s})ds+\sup_{\tilde{u} \in U}\{Y^{\nu\circ(t,u)\circ (\tau,\tilde{u}),k-1}_{\tau}-\ell(X^{\nu\circ (t,u)}_{\tau},\tilde{u},\tau)\} \mid  \mathcal{F}_{t}]\\
		&=\mathbb{E}[\int_{t}^{\tau^{*}}f(s,X^{\nu\circ (t,u)}_{s})ds+(Y^{\nu\circ(t,u)\circ (\tau^{*},\xi^{*}),k-1}_{\tau^{*}}- \ell(X^{\nu\circ (t,u)}_{\tau^{*}},\xi^{*},\tau^{*}))\chi_{\{\tau^{*} < T\}} \mid  \mathcal{F}_{t}]\\
		&=\mathbb{E}[\int_{t}^{T} f(s,X^{\nu\circ (t,u) \circ \bigcirc_{j=0}^{k} (\tau^{*}_{j},\xi^{*}_{j}) }_{s})ds-\sum_{i=0}^{k}  \ell(X^{\nu\circ (t,u) \circ \bigcirc_{j=0}^{k} (\tau^{*}_{j},\xi^{*}_{j})}_{\tau_{i}^{*}},\xi_{i}^{*},\tau_{i}^{*})\chi_{\{\tau^{*} < T\}} \mid  \mathcal{F}_{t}]
	\end{aligned}
\end{equation}
where $(\tau^{*}_{j},\xi^{*}_{j})_{j=0}^{n}$ is defined as in (\ref{nko}) and (\ref{nkj}). This is thus essentially the algorithm from the verification theorem, the situation differs due to the restricted number of interventions allowed in each step. 

To obtain the second equality we need what is inside of the Snell envelope to be continuous in $t$, adapted, of class [D] and what is inside of the sumpreum to be continuos in $u$. For the third we need the same to be true for all $m \le k-1$ and to use $(\ref{ohno})$. Since the Stiltjes integral has all the mentioned properties we restrict our attention to terms of the form,
\begin{equation}\label{daa}
	\sup_{u \in U}\{Y_{t}^{\hat{\nu} \circ (t,u),m}-\ell(X^{\hat{\nu}}_{t},u,t)\},
\end{equation}
for any $\hat{\nu}$. Note that in order for us to get continuity of $(\ref{daa})$ it is sufficient to have continuity in both $t$ and $u$ since $[0,T]$ and $U$ are both compact.

Suppose $(\ref{daa})$ and what is inside of the supremum is continuous in $t$ and $u$ respectively for some $m<k+1$ and all $m^{'}< m$. This means that we have (\ref{verf}) for $m+1$. Moreover given $\tilde{\nu}=(\tilde{\tau}_{j}\vee t,\tilde{\xi}_{j}) \in \mathcal{A}_{f}^{m+1}$ we set,
\begin{equation}
	\begin{aligned}
		&\tilde{Y}^{\hat{\nu}\circ (t,u) \circ \tilde{\nu}}_{t}=\mathbb{E}[\int_{t}^{T} f(s,X^{\hat{\nu}\circ (t,u) \circ \bigcirc_{j=0}^{m+1} (\tilde{\tau}_{j}\vee t,\tilde{\xi}_{j}) }_{s})ds\\
		&-\sum_{i=0}^{m+1}  \ell(X^{\hat{\nu}\circ (t,u) \circ \bigcirc_{j=0}^{m+1} (\tilde{\tau}_{j}\vee t,\tilde{\xi}_{j})}_{\tilde{\tau}_{i}},\tilde{\xi}_{i},\tilde{\tau}_{i}\vee t) \mid  \mathcal{F}_{t}],
	\end{aligned}
\end{equation}
which implies,
\begin{equation}
	\mid Y^{\hat{\nu} \circ (t,u),m+1}_{t}-Y^{\hat{\nu} \circ (\hat{t},\hat{u}),m+1}_{\hat{t}}\mid \le \sup_{\tilde{\nu} \in \mathcal{A}_{f}^{m+1}}\mid \tilde{Y}^{\hat{\nu} \circ (t,u),\tilde{\nu}}_{t}-\tilde{Y}^{\hat{\nu} \circ (\hat{t},\hat{u}),\tilde{\nu}}_{\hat{t}} \mid. 
\end{equation}
Hence,
\begin{equation}\label{jaj}
	\begin{aligned}
		&\mathbb{E}\mid Y^{\hat{\nu}\circ (t,u),m+1}_{t}-Y^{\hat{\nu}\circ (\hat{t},\hat{u}),m+1}_{\hat{t}} \mid^{6+2l}  \\
		&\le \mathbb{E}\esssup_{\tilde{\nu} \in \mathcal{A}_{f}^{m+1}}\mid \tilde{Y}^{\hat{\nu} \circ (t,u),\tilde{\nu}}_{t}-\tilde{Y}^{\hat{\nu} \circ (\hat{t},\hat{u}),\tilde{\nu}}_{\hat{t}} \mid^{6+2l}   \\
		&=\mathbb{E}\Big(\esssup_{\tilde{\nu}\in \mathcal{A}_{f}^{m+1}}\mid\mathbb{E}[\int_{t}^{T} f(s,X^{\hat{\nu}\circ (t,u) \circ \tilde{\nu}}_{s})ds-\sum_{i=0}^{m+1}  \ell(X^{\hat{\nu}\circ (t,u) \circ \tilde{\nu}}_{\tilde{\tau}_{i}\vee t},\tilde{\xi}_{i},\tilde{\tau}_{i}\vee t) \mid  \mathcal{F}_{t}]\\
		&-\mathbb{E}[\int_{\hat{t}}^{T} f(s,X^{\hat{\nu}\circ (\hat{t},\hat{u}) \circ \tilde{\nu} }_{s})ds-\sum_{i=0}^{m+1}  \ell(X^{\hat{\nu}\circ (\hat{t},\hat{u})  \circ \tilde{\nu}}_{\tilde{\tau}_{i}\vee \hat{t}},\tilde{\xi}_{i},\tilde{\tau}_{i}\vee \hat{t}) \mid  \mathcal{F}_{\hat{t}}]\mid^{6+2l}\Big) \\
		&\le\sup_{\tilde{\nu} \in \mathcal{A}_{f}^{m+1}}\mathbb{E}\Big( \mid \mathbb{E}[\int_{t}^{T} f(s,X^{\hat{\nu}\circ (t,u) \circ \tilde{\nu} }_{s})ds-\sum_{i=0}^{m+1}  \ell(X^{\hat{\nu}\circ (t,u) \circ \tilde{\nu}}_{\tilde{\tau}_{i}\vee t},\tilde{\xi}_{i},\tilde{\tau}_{i}\vee t) \mid  \mathcal{F}_{t}]\\
		&-\mathbb{E}[\int_{t}^{T} f(s,X^{\hat{\nu}\circ (t,u) \circ \tilde{\nu} }_{s})ds-\sum_{i=0}^{m+1}  \ell(X^{\hat{\nu}\circ (t,u) \circ \tilde{\nu}}_{\tilde{\tau}_{i}\vee t},\tilde{\xi}_{i},\tilde{\tau}_{i}\vee t) \mid  \mathcal{F}_{\hat{t}}]\\
		&+\mathbb{E}[\int_{t}^{T} f(s,X^{\hat{\nu}\circ (t,u) \circ\tilde{\nu} }_{s})ds-\sum_{i=0}^{m+1}  \ell(X^{\hat{\nu}\circ (t,u) \circ \tilde{\nu}}_{\tilde{\tau}_{i}\vee t},\tilde{\xi}_{i},\tilde{\tau}_{i}\vee t) \mid  \mathcal{F}_{\hat{t}}]   \\
		&-\mathbb{E}[\int_{\hat{t}}^{T} f(s,X^{\hat{\nu}\circ (\hat{t},\hat{u}) \circ \tilde{\nu} }_{s})ds-\sum_{i=0}^{m+1}  \ell(X^{\hat{\nu}\circ (\hat{t},\hat{u})  \circ \tilde{\nu}}_{\tilde{\tau}_{i}\vee \hat{t}},\tilde{\xi}_{i},\tilde{\tau}_{i}\vee \hat{t}) \mid  \mathcal{F}_{\hat{t}}] \mid^{6+2l}\Big)\\
		&\le C\sup_{\tilde{\nu} \in \mathcal{A}_{f}^{m+1}}\mathbb{E}\Big( \mid \mathbb{E}[\int_{t}^{T} f(s,X^{\hat{\nu}\circ (t,u) \circ \tilde{\nu} }_{s})ds-\sum_{i=0}^{m+1}  \ell(X^{\nu\circ (t,u) \circ \tilde{\nu}}_{\tilde{\tau}_{i}\vee t},\tilde{\xi}_{i},\tilde{\tau}_{i}\vee t) \mid  \mathcal{F}_{t}]\\
		&-\mathbb{E}[\int_{t}^{T} f(s,X^{\hat{\nu}\circ (t,u) \circ \tilde{\nu} }_{s})ds-\sum_{i=0}^{m+1}  \ell(X^{\hat{\nu}\circ (t,u) \circ \tilde{\nu}}_{\tilde{\tau}_{i}\vee t},\tilde{\xi}_{i},\tilde{\tau}_{i}\vee t) \mid  \mathcal{F}_{\hat{t}}]\mid^{6+2l}\Big)\\
		&+C\sup_{\tilde{\nu} \in \mathcal{A}_{f}^{m+1}}\mathbb{E}\Big( \mid\mathbb{E}[\int_{t}^{T} f(s,X^{\hat{\nu}\circ (t,u) \circ\tilde{\nu} }_{s})ds-\sum_{i=0}^{m+1}  \ell(X^{\hat{\nu}\circ (t,u) \circ \tilde{\nu}}_{\tilde{\tau}_{i}\vee t},\tilde{\xi}_{i},\tilde{\tau}_{i}\vee t) \mid  \mathcal{F}_{\hat{t}}]   \\
		&-\mathbb{E}[\int_{\hat{t}}^{T} f(s,X^{\hat{\nu}\circ (\hat{t},\hat{u}) \circ \tilde{\nu} }_{s})ds-\sum_{i=0}^{m+1}  \ell(X^{\hat{\nu}\circ (\hat{t},\hat{u})  \circ \tilde{\nu}}_{\tilde{\tau}_{i}\vee \hat{t}},\tilde{\xi}_{i},\tilde{\tau}_{i}\vee \hat{t}) \mid  \mathcal{F}_{\hat{t}}] \mid^{6+2l}\Big).
	\end{aligned}
\end{equation}
Which by the martingale representations theorem,
\begin{equation}
	\begin{aligned}
		&= C\sup_{\tilde{\nu} \in \mathcal{A}_{f}^{m+1}}\mathbb{E}\Big(\mid\mathbb{E}[M_{0}]+\int_{0}^{t}\mathcal{M}_{s}dB_s
		-\mathbb{E}[M_{0}]-\int_{0}^{\hat{t}}\mathcal{M}_{s}dB_s\mid^{6+2l}\Big)\\
		&+C\sup_{\tilde{\nu} \in \mathcal{A}_{f}^{m+1}}\mathbb{E}\Big( \mid \mathbb{E}[\int_{t}^{T} f(s,X^{\hat{\nu}\circ (t,u) \circ\tilde{\nu} }_{s})- f(s,X^{\hat{\nu}\circ (\hat{t},\hat{u}) \circ \tilde{\nu} }_{s})ds\\
		&+\sum_{i=0}^{m+1}  \ell(X^{\hat{\nu}\circ (\hat{t},\hat{u})  \circ \tilde{\nu}}_{\tilde{\tau}_{i}\vee \hat{t}},\tilde{\xi}_{i},\tilde{\tau}_{i}\vee \hat{t})-\ell(X^{\hat{\nu}\circ (t,u) \circ \tilde{\nu}}_{\tilde{\tau}_{i}\vee t},\tilde{\xi}_{i},\tilde{\tau}_{i}\vee t) \mid^{6+2l}\mathcal{F}_{\hat{t}}]\Big) 
	\end{aligned}
\end{equation}
\begin{equation}
	\begin{aligned}
		&\le C\sup_{\tilde{\nu} \in \mathcal{A}_{f}^{m+1}}\mathbb{E}\Big(\mid\int_{t}^{\hat{t}}\mathcal{M}_{s}dB_s\mid^{6+2l}\Big)\\
		&+C\sup_{\tilde{\nu} \in \mathcal{A}_{f}^{m+1}}\mathbb{E}\Big(  K_5(T-t)\sup_{s\in[t,T]}\norm{X^{\hat{\nu} \circ (t,u) \circ \tilde{\nu}}_{s}-X^{\hat{\nu} \circ (\hat{t},\hat{u}) \circ \tilde{\nu}}_{s}}^{6+2l}\\
		&+K_7\sum_{i=0}^{m+1} \norm{(X^{\hat{\nu}\circ (\hat{t},\hat{u})  \circ \tilde{\nu}}_{\tilde{\tau}_{i}\vee \hat{t}}-X^{\hat{\nu}\circ (t,u) \circ \tilde{\nu}}_{\tilde{\tau}_{i}\vee t},\tilde{\tau}_{i}\vee \hat{t}-\tilde{\tau}_{i}\vee t)}^{6+2l}\Big) \\
		& \le (t-\hat{t})^{2+l}C_{6+2l}\mathbb{E}[(\int_{0}^{T}\mid\mathcal{M}_{s}\mid^2ds)^{3+l}]+ C\norm{(t-\hat{t},u-\hat{u})}^{2+l},
	\end{aligned}
\end{equation}
where we obtain the last inequality using Burkholder-Davis-Gundy followed by Hölders inequality on the first term and Theorem \ref{dyn} on the second. Regarding $Y^{\nu\circ (t,u),m+1}_{t}$ as a random field indexed by $(t,u)$ puts us in position to use Kolmogorov continuity theorem. This in turn means that there exist a map $\chi:[0,T]\times\Omega\times U\to\mathbb{R}$ that is jointly continuous in $(t,u)$ such that $\chi(t,u)=Y^{\nu\circ (t,u),m+1}_t$, $\mathbb{P}$-a.s., for each $(t,u)\in[0,T]\times U$. On the other hand, a simple approximation routine shows that the $\mathbb{P}$-null set can be chosen independent of $(t,u)$ and $\mathbb{P}$-a.s.~continuity of $(t,u)\mapsto Y^{\nu\circ (t,u),m+1}_t$ follows. Hence we obtain continuity in $t$ of,
\begin{equation}
	\sup_{u \in U}\{Y_{t}^{\nu \circ (t,u),m+1}-\ell(X^{\nu}_{t},u,t)\},
\end{equation}
by arguing as above.
Finally, it also holds for $m=0$ by (\ref{jaj}).

Turning to adaptedness we use Corollary \ref{corsel}. We see that we have joint measurability of $Y_{t}^{\nu \circ (t,u),m}(\omega)$ for a given $t$ by looking at (\ref{verf}). This follows since the conditional is $\mathcal{F}_{t}$-measurable and we have continuity in $u$ for any version of the conditional.

Thus there is a $u^{*}_{t}(\omega)$ which is $\mathcal{F}_{t}$-measurable for which the supremum is attained. This gives us
\begin{equation}
	\sup_{u \in U}\{Y_{t}^{\nu \circ (t,u),m+1}-\ell(X^{\nu}_{t},u,t)\}=Y_{t}^{\nu \circ (t,u^{*}_{t}(\omega)),m+1}-\ell(X^{\nu}_{t},u^{*}_{t}(\omega),t),
\end{equation}
which is $\mathcal{F}_{t}$-measurable.

To establish that what we take the Snell envelope of is of class [D] we proceed as follows. Observe that $Y^{\nu\circ (t,\hat{u}),0}_{t}\le Y^{\nu\circ (t,\hat{u}),1}_{t}$ since the latter equals the former with $\tau=T$.

Suppose now that,
\begin{equation}
	Y^{\nu\circ (t,\hat{u}),k}_{t} \le Y^{\nu\circ (t,\hat{u}),k+1}_{t}
\end{equation}
holds for some $k$. Since,
\begin{equation}
	\begin{aligned}
		&Y^{\nu\circ (t,\hat{u}),k+2}_{t}\\
		&=\esssup _{\tau \in \mathcal{T}_{t}}\mathbb{E}[\int_{t}^{\tau}f(s,X^{\nu\circ (t,\hat{u})}_{s})ds\\
		&+\sup_{u\in U}\{ Y^{\nu\circ (t,\hat{u}) \circ (\tau,u),k+1}_{\tau}- \ell(X^{\nu}_{\tau},u,\tau)\}\chi_{\{\tau<T\}}\mid  \mathcal{F}_{t}] \\
		&\ge \esssup _{\tau \in \mathcal{T}_{t}}\mathbb{E}[\int_{t}^{\tau}f(s,X^{\nu\circ (t,\hat{u})}_{s})ds\\
		&+\sup_{u\in U}\{ Y^{\nu \circ (t,\hat{u})\circ (\tau,u),k}_{\tau}- \ell(X^{\nu}_{\tau},u,\tau)\}\chi_{\{\tau<T\}}\mid  \mathcal{F}_{t}]\\
		&=Y^{\nu\circ (t,\hat{u}),k+1}_{t}
	\end{aligned}
\end{equation}
we obtain monotonicity in $k$.  Next since $\ell> K_{6}> 0$ we have,
\begin{equation}
	\begin{aligned}
		&Y^{\nu\circ (t,\hat{u}),k-1}_{t}\\
		&=\esssup _{\tau \in \mathcal{T}_{t}} \mathbb{E}[\int_{t}^{\tau}f(s,X^{\nu\circ (t,\hat{u})}_{s})ds+\sup_{u \in U}\{Y^{\nu\circ (t,\hat{u}) \circ (\tau,u)}- \ell(X^{\nu}_{\tau},u,\tau)\}\mid  \mathcal{F}_{t}] \\
		&\le \mathbb{E}[\int_{t}^{T}\sup_{\hat{\nu}}f(s,X^{\hat{\nu}}_{s})ds\mid  \mathcal{F}_{t}].
	\end{aligned} 
\end{equation}
We obtain the bound for $ \mathcal{S}^p_{c}$ for all $m<k$ using the growth assumptions on $f$ and Theorem \ref{dyn}. We can therefore conclude that they are of class [D].

With this at hand we have representation (\ref{verf}) for any $k$ which by (\ref{jaj}) implies continuity of $Y_{t}^{\nu \circ (t,u),k}$ by the same reasoning.\begin{flushright} $\square$ \end{flushright}

\begin{Lem}
	Each member of  $\{Y^{\nu,k}_{s}\}_{\nu \in \mathcal{A}_{f} }$ belong to $\mathcal{S}^{p}_{c}$ for all $k$
\end{Lem}
\emph{Proof.} We proceed by induction. For $k=0$ a closed martingale, which in a Brownian filtration has an a.s continuous version. The second term is a Stiljtjes integral which is also continuous, thus the statement is true for $k=0$.

Now suppose the statement is true for some $k$ and for the whole family i.e that 
\begin{equation}
	Y^{\nu,k}_{s}=\esssup_{\tau \in \mathcal{T}_{s}}\mathbb{E}[\int_{s}^{\tau}f(t,X^{\nu}_{t})dt+\sup_{u\in U}\{Y^{\nu\circ (\tau, u),k-1}_{\tau}- \ell(X^{\nu}_{\tau},u,\tau)\}\mid  \mathcal{F}_{s}]
\end{equation}
exists and has a continuous version for any $\nu$.

As $Y^{\nu,k+1}_{s}$ is the Snell envelope of the process 
\begin{equation}
	(\int_{s}^{t}f(t,X^{\nu}_{t})dt+\sup_{u\in U}\{Y^{\nu\circ (t,u),k}_{t}- \ell(X^{\nu}_{t},u,t)) \}\chi_{\{\tau<T\}})_{t\ge 0}
\end{equation}
we will establish that the latter is of class [D], continuous and adapted in order to use Theorem \ref{Snell} (iii) to conclude that $Y^{\nu,n+1}_{s}$ indeed exist and is sufficiently regular.

The first term is obviously continuous and adapted. Moving on to the second term, i.e. $\sup_{u\in U}\{Y^{\nu\circ (t,u),k}_{t}- \ell(X^{\nu}_{t},u,t)\}$ we apply Lemma \ref{jcts} and argue as in (\ref{daa}) to obtain continuity, to get adaptedness we argue exactly as in the previous lemma.

That $Y^{\nu,k}_{t}$ satisfy the bound for  $\mathcal{S}^{p}_{c}$ and thus also belong to class [D] follows by the same reasoning as at the end of the proof of Lemma \ref{jcts}. \begin{flushright} $\square$ \end{flushright}

Recalling that 
\begin{equation}
	Y^{\nu,k}_{t}\le Y^{\nu,k+1}_{t}
\end{equation}
and notice that
\begin{equation}
	\begin{aligned}
		&Y^{\nu,k}_{t}=\esssup _{\tau \in \mathcal{T}_{t}} \mathbb{E}[\int_{t}^{\tau}f(s,X^{\nu}_{s})ds+\sup_{u \in U}\{Y^{\nu \circ (\tau,u),k-1}- \ell(X^{\nu}_{\tau},u,\tau)\}\chi_{\{\tau<T\}}\mid  \mathcal{F}_{t}] \\
		&\le \mathbb{E}[\int_{t}^{T}\sup_{\hat{\nu}}f(s,X^{\hat{\nu}}_{s})ds\mid  \mathcal{F}_{t}]<\infty
	\end{aligned}
\end{equation}
due to growth conditions on $f$ as well as Theorem \ref{dyn}, we define $\tilde{Y}^{\nu}_{t}=\lim_{k}Y^{\nu,k}_{t}$ for each $t$.
In the next Lemma we establish that this convergence is uniform when considered on the function  $Y^{\nu \circ (t,u),k}_{t}-\ell(X^{\nu}_{t},u,t)$ of $(t,u)$.
\begin{Lem}\label{la}
	$Y^{\nu \circ (t,u),k}_{t}-\ell(X^{\nu}_{t},u,t)$ converges uniformly as a function of $(t,u)$ as $k\to\infty$  for a.e. $\omega$.
	\begin{proof}
		We start by observing that
		\begin{equation}
			\begin{aligned}
				&\sup_{(t,u)\in [0,T]\times U}\mid Y^{\nu \circ (t,u),k}_{t}+\mathbb{E}[\sum_{i=0}^{k}  \ell(X^{\nu}_{\tau_{i}^{*}},u^{*},\tau_{i}^{*})\mid  \mathcal{F}_{t}] \mid \\
				&=\sup_{(t,u)\in [0,T]\times U} \mathbb{E}[\int_{t}^{T} \mid f(s,X^{\nu\circ (t,u) \circ \bigcirc_{j=0}^{k} (\tau^{*}_{j},u^{*}_{j}) }_{s})\mid ds \mid  \mathcal{F}_{t}]\\
				&\le C \cdot \sup_{(t,u)\in [0,T]\times U}  \mathbb{E}[\int_{t}^{T}\sup_{s\in [0,T]}C+\mid X^{\nu\circ (t,u) \circ \bigcirc_{j=0}^{k}  (\tau^{*}_{j},u^{*}_{j}) }_{s}\mid^{2} ds \mid  \mathcal{F}_{t}] \\
				&\le C\cdot \sup_{(t,u)\in [0,T]\times U}  \mathbb{E}[T \cdot \sup_{s\in [0,T]}\mid C+ X^{\nu\circ (t,u) \circ \bigcirc_{j=0}^{k}  (\tau^{*}_{j},u^{*}_{j}) }_{s}\mid^{2} \mid  \mathcal{F}_{t}]\\ 
			\end{aligned}
		\end{equation}
		Moreover, since
		\begin{equation}
			\begin{aligned}
				&\mathbb{E}\sup_{(t,u)\in [0,T]\times U}\mid Y^{\nu \circ (t,u),k}_{t}+\mathbb{E}[\sum_{i=0}^{k}  \ell(X^{\nu}_{\tau_{i}^{*}},u^{*},\tau_{i}^{*})\mid  \mathcal{F}_{t}]\mid  \\
				&\le C \mathbb{E}\sup_{t\in [0,T]}\sup_{u\in U}\sup_{r\in [t,T]}\mathbb{E}[T \cdot \sup_{s\in [0,T]}\mid C+ X^{\nu\circ (r,u) \circ \bigcirc_{j=0}^{k}  (\tau^{*}_{j},u^{*}_{j}) }_{s}\mid^{2} \mid  \mathcal{F}_{t}] \\  
				&\le C \cdot T \cdot\mathbb{E} \mathbb{E}[ \sup_{s\in [0,T]}\sup_{\hat{\nu}\in \mathcal{A}}\mid C+ X^{\nu\circ \hat{\nu} \circ\bigcirc_{j=0}^{k}  (\tau^{*}_{j},u^{*}_{j}) }_{s}\mid^{2} \mid  \mathcal{F}_{T}]  \\
				&=C\cdot \mathbb{E}\sup_{s\in [0,T]}\sup_{\hat{\nu}\in \mathcal{A}}\mid C+ X^{\nu\circ \hat{\nu} \circ \bigcirc_{j=0}^{k}  (\tau^{*}_{j},u^{*}_{j}) }_{s}\mid^{2} <C,
			\end{aligned}
		\end{equation}
		by Doobs inequality and a similar reasoning as in Theorem \ref{dyn}, we conclude that
		\begin{equation}
			\sup_{(t,u)\in [0,T]\times U}\mid Y^{\nu \circ (t,u),k}_{t}+\mathbb{E}[\sum_{i=0}^{k}  \ell(X^{\nu}_{\tau_{i}^{*}},u^{*},\tau_{i}^{*}) \cdot \chi_{\{\tau^{*}_{i} < T\}}\mid  \mathcal{F}_{t}] \mid < K(\omega).
		\end{equation}
		for some $\mathbb{P}$-a.s. finite $\mathcal{F}_{T}$-measurable random variable $K(\omega)$. Hence for any $\omega \in \Omega \setminus \mathcal{N}$ for some nullset $\mathcal{N}$ we have
		
		\begin{equation}\label{ka} 
			\begin{aligned}
				&\mathbb{E}[\sum_{i=0}^{k}  \ell(X^{\nu}_{\tau_{i}^{*}},u^{*},\tau_{i}^{*}) \cdot \chi_{\{\tau^{*}_{i} < T\}}\mid  \mathcal{F}_{t}] \le K(\omega)- Y^{\nu \circ (t,u),k}_{t}  \\
				&\qquad \qquad \qquad \qquad \qquad \qquad \iff \\
				& c\cdot k^{'}\cdot\mathbb{E}[\chi_{\{\tau_{k^{'}}<T\}} \mid \mathcal{F}_{t}]  \le K(\omega) - Y^{\nu \circ (t,u),k}_{t}  \\
				&\qquad \qquad \qquad \qquad \qquad \qquad \iff \\
				&\mathbb{E}[\chi_{\{\tau_{k^{'}}<T\}} \mid  \mathcal{F}_{t}] \le \frac{ K(\omega) - Y^{\nu \circ (t,u),0}_{t}}{c\cdot k^{'}}\le \frac{2K(\omega)}{c\cdot k^{'}}
			\end{aligned}
		\end{equation}
		
		Next for $k^{'}\le \hat{k}$ we define
		\begin{equation}
			\begin{aligned}
				&Y^{\nu \circ (t,u),\hat{k},k^{'}}_{t}\\
				&:=\mathbb{E}[\int_{t}^{T} f(s,X^{\nu\circ (t,u) \circ \bigcirc_{j=0}^{\hat{k}} (\tau^{*}_{j\wedge k^{'}},u^{*}_{j\wedge k^{'}}) }_{s})ds\\
				&-\sum_{i=0}^{\hat{k}}  \ell(X^{\nu\circ (t,u) \circ \bigcirc_{j=0}^{\hat{k}} (\tau^{*}_{j},u^{*}_{j})}_{\tau_{i}^{*}},u_{i}^{*},\tau_{i}^{*}) \mid  \mathcal{F}_{t}] \\
				&\le Y^{\nu \circ (t,u),k^{'}}_{t} \le Y^{\nu \circ (t,u),\hat{k}}_{t}.
			\end{aligned}
		\end{equation}
		
		Since the truncation only change the control when $\tau_{k^{'}+1}<T$ we get
		\begin{equation}
			\begin{aligned}
				&\sup_{(t,u)\in [0,T]\times U}\mid Y^{\nu \circ (t,u),\hat{k}}_{t}-\ell(X^{\nu}_{t},u,t)-Y^{\nu \circ (t,u),k^{'}}_{t}+\ell(X^{\nu}_{t},u,t)\mid \\
				&\le \sup_{(t,u)\in [0,T]\times U}\mid Y^{\nu \circ (t,u),\hat{k}}_{t}-Y^{\nu \circ (t,u),\hat{k},k^{'}}_{t}\mid\\
				&\le \sup_{(t,u)\in [0,T]\times U}\mid\mathbb{E}[\chi_{\{\tau_{k^{'}+1}<T\}}(\int_{t}^{T} f(s,X^{\nu\circ (t,u) \circ \bigcirc_{j=0}^{\hat{k}} (\tau^{*}_{j},u^{*}_{j}) }_{s})ds\\
				&-\sum_{i=0}^{\hat{k}} \ell(X^{\nu\circ (t,u) \circ \bigcirc_{j=0}^{\hat{k}} (\tau^{*}_{j},u^{*}_{j}) }_{\tau_{i}^{*}},u_{i}^{*},\tau_{i}^{*})-\int_{t}^{T} f(s,X^{\nu\circ (t,u) \circ \bigcirc_{j=0}^{\hat{k}} (\tau^{*}_{j\wedge k^{'}},u^{*}_{j\wedge k^{'}}) }_{s})ds\\
				&+\sum_{i=0}^{\hat{k}} \ell(X^{\nu\circ (t,u) \circ \bigcirc_{j=0}^{\hat{k}} (\tau^{*}_{j\wedge k^{'}},u^{*}_{j\wedge k^{'}}) },u_{i}^{*},\tau_{i}^{*})) \mid  \mathcal{F}_{t}] \mid\\
				&\le \sup_{(t,u)\in [0,T]\times U} \mid\mathbb{E}[\chi_{\{\tau_{k^{'}+1}<T\}}(\int_{t}^{T} f(s,X^{\nu\circ (t,u) \circ \bigcirc_{j=0}^{\hat{k}}  (\tau^{*}_{j},u^{*}_{j}) }_{s})ds\\
				&-\int_{t}^{T}f(s,X^{\nu\circ (t,u) \circ \bigcirc_{j=0}^{\hat{k}} (\tau^{*}_{j\wedge k^{'}},u^{*}_{j\wedge k^{'}}) }_{s})ds \mid  \mathcal{F}_{t}] \mid\\
				&\le \sup_{(t,u)\in [0,T]\times U}\mathbb{E}[\chi_{\{\tau_{k^{'}+1}<T\}} \mid  \mathcal{F}_{t}]^{\frac{1}{q}} \cdot \\
				&\mathbb{E}[\int_{t}^{T}\mid f(s,X^{\nu\circ (t,u) \circ \bigcirc_{j=0}^{\hat{k}}(\tau^{*}_{j},u^{*}_{j}) }_{s})-f(s,X^{\nu\circ (t,u) \circ \bigcirc_{j=0}^{\hat{k}} (\tau^{*}_{j\wedge k^{'}},u^{*}_{j\wedge k^{'}}) }_{s})\mid^{p} ds\mid  \mathcal{F}_{t}]^{\frac{1}{p}}\\
				&\le \frac{C(\omega)}{(k^{'})^{\frac{1}{q}}}
			\end{aligned}
		\end{equation}
		by applying Hölder inequality and (\ref{ka}). Taking the limit in $k^{'}$ yields that the sequence is a.s. uniformly Cauchy which in turn gives the needed convergence.
	\end{proof}
\end{Lem}
Recall Definition \ref{lo} of a verification family, we are now in position to establish that such a family indeed exists.
\begin{proof}[Proof of Theorem \ref{Thm:neat}]
	Recall once more that
	\begin{equation}
		Y^{\nu,k}_{t}\le Y^{\nu,k+1}_{t}
	\end{equation}
	In order to apply Theorem \ref{Snell} (iii) to obtain continuity and (i) for its pointwise limit, we need to establish that what is inside the Snell envelope is of class [D] and that $\sup_{u\in U}\{\tilde{Y}^{\nu \circ (s,u)}_{s}-\ell( X,u,s)\}$ is a.s. continuous and adapted, which is (ii).
	Starting with the latter we have, for any $s$,
	\begin{equation}\label{lal}
		\begin{aligned}
			\lim_{k}\sup_{u\in U}\{Y^{\nu \circ (s,u),k}_{s}-\ell( X^{\nu },u,s)\}=\sup_{u\in U}\{\tilde{Y}^{\nu \circ (s,u)}_{s}-\ell( X^{\nu },u,s)\}
		\end{aligned}
	\end{equation}
	due to Lemma \ref{la}.
	
	Adaptedness follows since the above is a limit of processes that are adapted by Lemma \ref{jcts} and continuity follows by using Lemma \ref{la} again as well as Lemma \ref{nini}.
	
	That $\tilde{Y}^{\nu}_{t}$ and what we take the Snell envelope of satisfy the bound for  $\mathcal{S}^{p}_{c}$  follows, again, by the same reasoning as at the end of the proof of Lemma \ref{jcts}. And thus the latter also belong to class [D].
	
	We conclude that (ii) holds and as a consequence we obtain (i) and continuity.

	The last property is obtained by the fact that upper semi-continuity is preserved under uniform convergence.
\end{proof}	

\section{Applications to SDDEs, Markovian case and a numerical example}
\subsection{SDDEs and the Markovian case}
In this section we provide an application of our main result. In particular, we consider the special case where the state dynamic, in addition to the value of the state, also depends on the state translated backward with a fixed delay. Note that such systems are necessarily non-Markovian. There is a wide variety of systems in which some time is required for the control to reach it's full effect on the state and hence are subject to such delays. Work on these types of problems has been carried out in e.g \cite{d},\cite{dd}, \cite{ddd} and \cite{dddd}.

A textbook example is the well known delayed feedback which is usually illustrated by the following diagram
\begin{figure}[h!]
	\centering
	\includegraphics[width=0.75\textwidth]{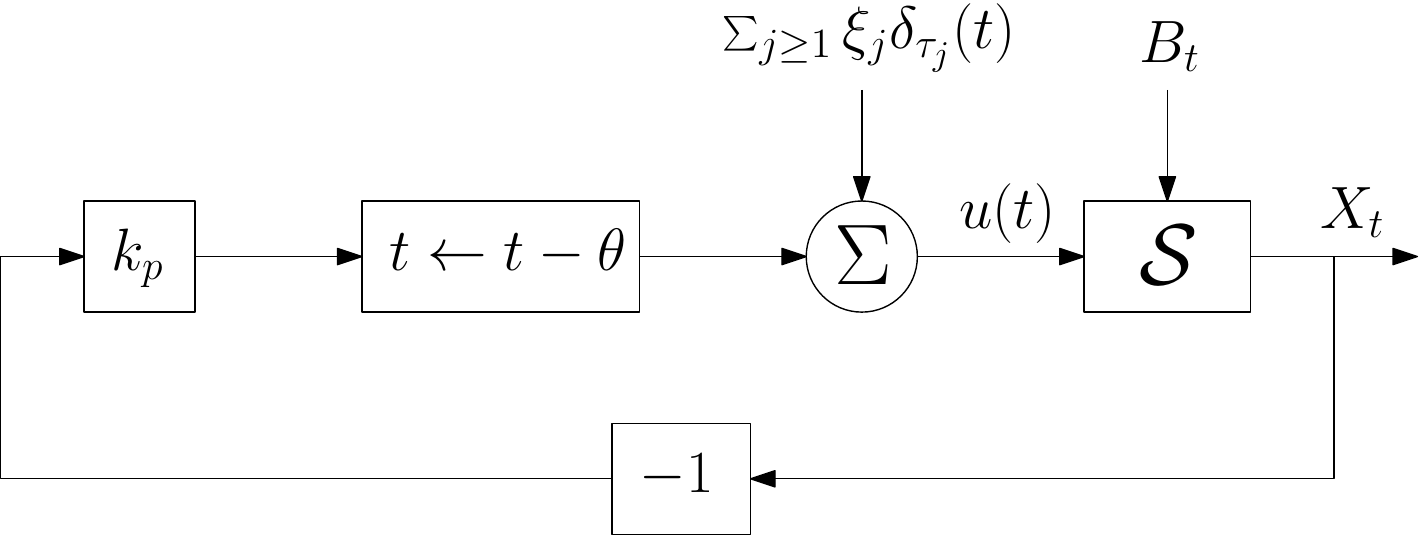}
	\caption{System used in the numerical example.}
	\label{fig:system}
\end{figure}

In other words, we add an impulse control to a proportional feedback system with delay $\theta$ in the continuous control actuation.

With this in mind, similar to above, we introduce,

\begin{Ass} 
	$(i)\,a:[0,T] \times \mathbb{R}^d \times \mathbb{R}^d \rightarrow \mathbb{R}^d$ and 
	$b:[0,T] \times \mathbb{R}^d \times \mathbb{R}^d \rightarrow \mathbb{R}^{d\times d}$ where  $a(t,0,0)$ and $b(t,0,0)$ are continuous in $t$ and the components satisfy 
	\begin{equation}
		\mid a_{i}(t,x,y)-a_{i}(t,\hat{x},\hat{y})\mid  \le K_{1}(\mid x-\hat{x}\mid+\mid y-\hat{y}\mid),
	\end{equation}

	\begin{equation}
		\mid b_{i,j}(t,x,y)-b_{i,j}(t,\hat{x},\hat{y})\mid  \le K_{2}(\mid x-\hat{x}\mid+\mid y-\hat{y}\mid),
	\end{equation}
	
	$K_{1},K_{2}$ being constants.
	
	$(ii)\,\Gamma:\mathbb{R}^n \times U \rightarrow  \mathbb{R}^n$ satisfy
	\begin{equation}
		\norm{\Gamma(x,u)}\le C \vee \norm{x} \, \text{and} \, 
	\end{equation}
	\begin{equation}
		\norm{\Gamma(x,u)-\Gamma(y,v)} \le \norm{(x,u)-(y,v)} \, \text{for all u,v}  \in U \text{and} \, x,y \in \mathbb{R}^n
	\end{equation}
\end{Ass}
and set
\begin{equation}
	\begin{aligned}
		&dX^{\alpha,\nu,0}_{t}=\alpha(0)+a(t,X^{\alpha,\nu,0}_{t},X^{\alpha,\nu,0}_{t-\delta})dt+b(t,X^{\alpha,\nu,0}_{t},X^{\alpha,\nu,0}_{t-\delta})dB_{t} \quad 0 \le t \le T \\
		&dX^{\alpha,\nu,0}_{t}=\alpha(t), \,  t\in [-\delta,0]
	\end{aligned}
\end{equation}
given some uniformly bounded function $\alpha \in \mathcal{D}$. Then, recursively define
\begin{equation}
	\begin{aligned}
		&dX^{\alpha,\nu,j}_{t}=a(t,X^{\alpha,\nu,j}_{t},X^{\alpha,\nu,j}_{t-\delta})dt+b(t,X^{\alpha,\nu,j}_{t},X^{\alpha,\nu,j}_{t-\delta})dB_{t} \quad \tau_{j}< t \le T\\
		&X^{\alpha,\nu,j}_{\tau_{j}}=\Gamma(X^{\alpha,\nu,j-1}_{\tau_{j}},\xi_{j}) \\
		&X^{\alpha,\nu,j}_{t}=X^{\alpha,\nu,j-1}_{t} \quad 0 \le t  < \tau_{j}.
	\end{aligned}
\end{equation}
Finally, to obtain our controlled state we put $\limsup_{j\rightarrow \infty} X^{\alpha,\nu,j}=X^{\alpha,\nu}$.

These constrains are clearly stronger than the ones in Assumptions \ref{A1} and hence by our above result we have existence and characterisation of an optimal control when the underlying dynamics depends on a delayed state.

By letting the delay $\delta=0$, we obtain the required assumptions in the Markovian setting. While the above condition on the dynamics are slightly stronger than those of \cite{tysk}, they cover the most typical and easily verified conditions. Moreover, we do not need to make any assumptions on the optimal control.
\subsection{A numerical example}
The several distinct frameworks in which we can consider impulse problems give us different options on how to numerically compute the optimal control. Doing so for non-Markovian problems is notoriously difficult, due to the high dimension of the state space. In recent years, the latter has attracted considerable attention due to its importance in machine learning and artificial intelligence. Below, we incorporate a recently proposed method using deep neural networks to make a Markov approximation of a non-Markovian system more tractable.

Let the system $\mathcal S$ above be given by
\begin{align}
	X_t=X_0+\int_0^t(a X_r+bu(r))dr+W_t,
\end{align}
leading to the following impulsively controlled SDDE representation
\begin{align}
	X_t=X_0+\int_0^t(a X_r-k_pX_{r-\theta})dr+W_t+\sum_{j\geq 1}\chi_{[\tau_j\leq t]}\beta_j.
\end{align}
Our aim is to find an impulse control that minimizes
\begin{align}
	J(u):=\mathbb{E}\Big[X_1^2+\int_0^1 X^2_r dr+0.1\sum_{j\geq 1}(1+\beta_j^2)\Big],
\end{align}
when $X_s=0$ for $s\in [-\theta,0]$, $k_p=a=1$, $\theta=0.05$ and $U=[-2,2]$.

To obtain numerical approximations of stochastic systems with delays one can, as mentioned, apply a particular time-discretization of the problem (for convergence properties see e.g. \cite{KushnerMCdelay}). This renders a finite dimensional model of the system. In particular, a discretization step $\Delta t$ (we assume that $\theta$ is a multiple of $\Delta t$) gives us a state-space dimension of $\theta/\Delta t+1$. The resulting Markovian discrete-time problem can then be solved by standard methods. 

Due to the potentially high dimension of the state space, we have resorted to a dynamic programming  approach based on value function approximation by neural networks recently proposed in~\cite{HureDeepNN}, in particular the one called Hybrid-Now. Note that even if our main result establishes one of the sufficient conditions in this approximation, one has to make sure that the $\argmin$ in the neural network approximation can be computed.

Solving the problem with $\Delta t=0.01$ we get the sample trajectories in Figure~\ref{fig:samp-traj}.
\begin{figure}[b]
	\centering
	\includegraphics[width=0.75\textwidth]{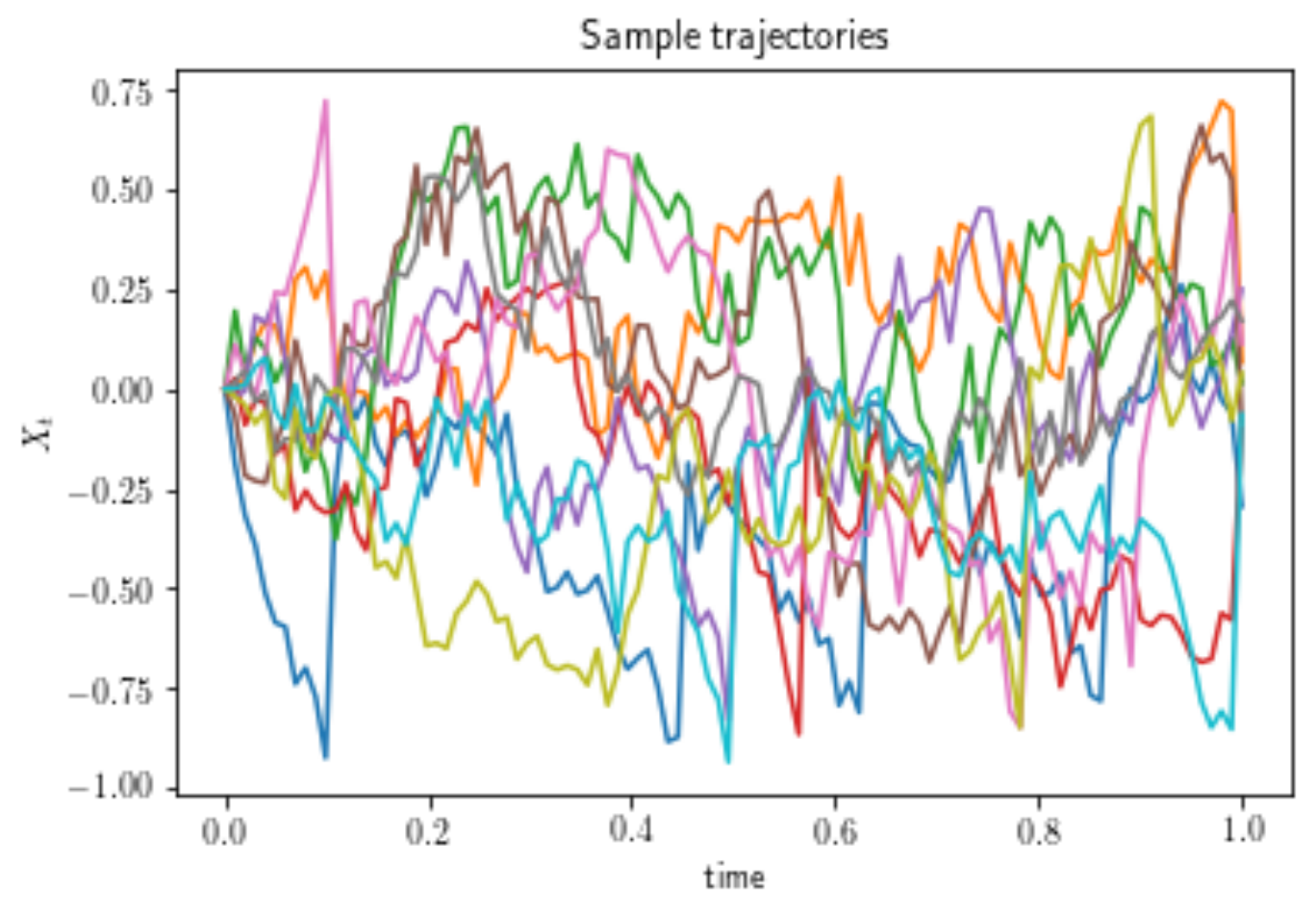}
	\caption{Sample trajectories of the optimally controlled process.}
	\label{fig:samp-traj}
\end{figure}
The value function is plotted in Figure~\ref{fig:vf}.
\begin{figure}[b]
	\centering
	\includegraphics[width=0.75\textwidth]{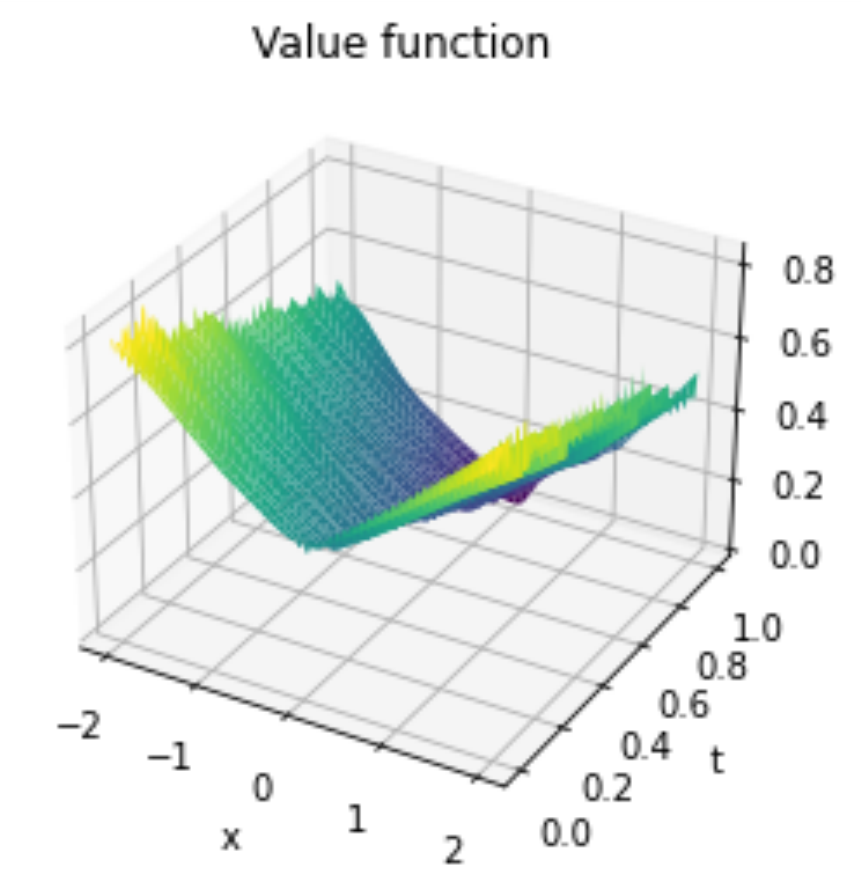}
	\caption{Value function in $X_s=x$ for $s\in [t-\theta,t]$.}
	\label{fig:vf}
\end{figure}
An optimal control is plotted in Figure~\ref{fig:oc}.
\begin{figure}[b]
	\centering
	\includegraphics[width=0.75\textwidth]{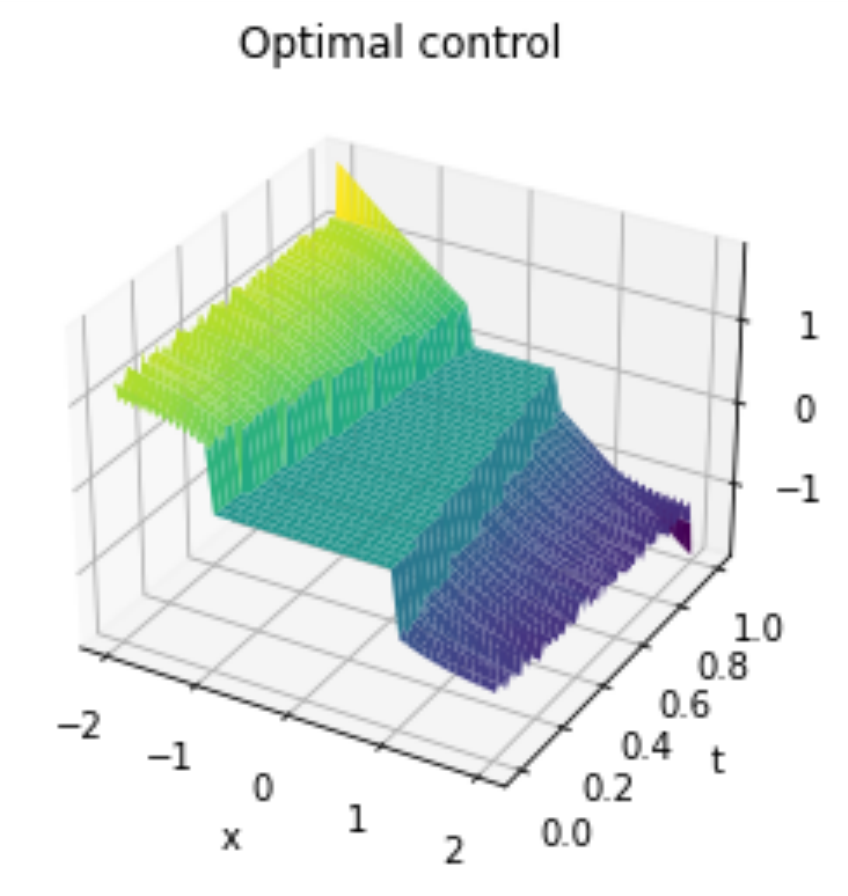}
	\caption{Optimal control in $X_s=x$ for $s\in [t-\theta,t]$.}
	\label{fig:oc}
\end{figure}
\section*{Acknowledgments}
This work was supported by the Swedish Energy Agency through grant number 42982-1.
\bibliographystyle{unsrt}
\bibliography{siam_ref}
\end{document}